\documentclass[11pt,reqno]{amsart}
\usepackage[utf8]{inputenc}
\usepackage{stmaryrd}
\usepackage{bbm,tikz}
\usepackage{mathrsfs}
\usepackage{enumerate}
\usepackage{latexsym,amsxtra}
\usepackage{dsfont}
\usepackage{slashed}
\usepackage[all]{xy}
\usepackage{amscd,graphics}
\usepackage{amsmath,amsfonts,amsthm,amssymb}
\usepackage{latexsym,amsmath}
\usepackage{graphicx,psfrag}
\usepackage{mathabx}
\usepackage{hyperref}
\usepackage{fancyhdr}
\usepackage{slashed}
\usepackage{pinlabel}

\newcommand{\R}{\mathcal R}
\newcommand{\p}{\partial}
\newcommand{\M}{\mathcal M}
\newcommand{\A}{\mathcal A}
\newcommand{\G}{\mathcal G}
\newcommand{\B}{\mathcal B}
\newcommand{\Z}{\mathbb Z}
\newcommand{\Q}{\mathbb Q}

\newcommand{\CC}{\mathcal C}
\newcommand{\BL}{\mathcal BL}
\renewcommand{\H}{\mathcal H}
\renewcommand{\S}{\mathcal S}
\newcommand{\D}{\mathcal D}

\newcommand{\s}{\mathfrak{s}}
\newcommand{\sign}{\operatorname{sign}}
\newcommand{\tr}{\operatorname{tr}}
\newcommand{\scal}{\operatorname{scal}}

\newcommand{\SF}{\operatorname{SF}}
\newcommand{\ind}{\operatorname{ind}}
\newcommand{\Int}{\operatorname{Int}}
\newcommand{\Lef}{\operatorname{Lef}\,}
\newcommand{\hmred}{HM^{\red}}

\newcommand{\hfhat}{\widehat{\operatorname{HF}}}

\renewcommand{\phi}{\varphi}
\newcommand{\id}{\operatorname{id}}

\newcommand{\su}{\mathfrak{su}}
\renewcommand{\sf}{\operatorname{sf}}
\newcommand{\im}{\operatorname{im}}
\newcommand{\ad}{\operatorname{ad}}
\newcommand{\Ad}{\operatorname{Ad}}
\newcommand{\Hom}{\operatorname{Hom}}
\newcommand{\Fix}{\operatorname{Fix}}
\newcommand{\coker}{\operatorname{coker}}
\newcommand{\spinc}{{\rm spin}^c}
\newcommand{\ab}{\operatorname{ab}}
\newcommand{\red}{\operatorname{red}}
\newcommand{\irr}{\operatorname{irr}}
\renewcommand{\sf}{\operatorname{sf}}
\newcommand{\gr}{\operatorname{gr}}
\newcommand{\lsw}{\lambda_{\,\rm{SW}}}
\newcommand{\lfo}{\lambda_{\,\rm{FO}}}
\newcommand{\etasig}{\eta_{\sign}}
\newcommand{\etasiga}{\eta_{\sign,\alpha}}

\newtheorem{thm}{Theorem}[section]

\newtheorem{lem}[thm]{Lemma}
\newtheorem{cor}[thm]{Corollary}
\newtheorem{pro}[thm]{Proposition}

\newtheorem{introthm}{Theorem}
\newtheorem{introcor}[introthm]{Corollary}
\newtheorem{introconj}[introthm]{Conjecture}

\theoremstyle{definition}
\newtheorem{defi}[thm]{Definition}

\newtheorem{rmk}[thm]{Remark}
\newcommand{\largecup}{\mbox{\Large $\cup$}}
\newcommand{\largecap}{\mbox{\Large $\cap$}}

\title{On the monopole Lefschetz number of finite order diffeomorphisms}
\thanks{The first author was partially supported by the NSF Grant DMS-1707857, the second author was partially supported by NSF Grant DMS-1811111 and a Simons Fellowship, and the third author was partially supported by Collaboration Grant  \#\,426269 from the Simons Foundation}
\author[Jianfeng Lin]{Jianfeng Lin}
\address{Department of Mathematics\newline\indent University of California, San Diego \newline\indent La Jolla, CA 92093}
\email{\rm{jil063@ucsd.edu}}
\author[Daniel Ruberman]{Daniel Ruberman}
\address{Department of Mathematics, MS 050\newline\indent Brandeis
University \newline\indent Waltham, MA 02454}
\email{\rm{ruberman@brandeis.edu}}
\author[Nikolai Saveliev]{Nikolai Saveliev}
\address{Department of Mathematics\newline\indent
University of Miami \newline\indent PO Box 249085
\newline\indent Coral Gables, FL 33124}
\email{\rm{saveliev@math.miami.edu}}

\begin{document}
\begin{abstract}
Let $K$ be a knot in an integral homology 3-sphere $Y$, and $\Sigma$ the corresponding $n$-fold cyclic branched cover. Assuming that $\Sigma$ is a rational homology sphere (which is always the case when $n$ is a prime power), we give a formula for the Lefschetz number of the action that the covering translation induces on the reduced monopole homology of $\Sigma$. The proof relies on a careful analysis of the Seiberg--Witten equations on 3-orbifolds and of various $\eta$-invariants. We give several applications of our formula: (1) we calculate the Seiberg--Witten and Furuta--Ohta invariants for the mapping tori of all semi-free actions of $\Z/n$ on integral homology 3-spheres; (2) we give a novel obstruction (in terms of the Jones polynomial) for the branched cover of a knot in $S^3$ being an $L$-space; (3) we give a new set of knot concordance invariants in terms of the monopole Lefschetz numbers of covering translations on the branched covers. 
\end{abstract}
\maketitle

\section{Introduction}
Monopole Floer homology, as defined by Kronheimer and Mrowka \cite{kronheimer-mrowka:monopole} and Fr\o yshov \cite{froyshov:monopole}, is a powerful invariant of 3-manifolds. Orientation preserving self-diffeomorphisms of the 3-manifold act on the monopole Floer homology. We initiated a study of this action in \cite{LRS2}, where we calculated the monopole Lefschetz number of an involution on a rational homology 3-sphere that makes it into a double branched cover over a link in the 3-sphere. In this paper we continue this study and extend our calculations to all finite order diffeomorphisms making a rational homology 3-sphere into a branched cover of a knot in an integral homology 3-sphere. Note that the case of free diffeomorphisms of finite order was dealt with in \cite{RS:jktr}.

Given a knot $K$ in an integral homology 3-sphere $Y$ and an integer $n\ge 2$, consider the $n$--fold cyclic branched cover $\Sigma \to Y$ with branch set $K$ and denote by $\tau: \Sigma \to \Sigma$ the covering translation. Assume that $\Sigma$ is a rational homology 3--sphere (which is always true when $n$ is a prime power). Observe that $\Sigma$ admits a unique $\spinc$ structure $\s$ such that $\tau_*(\s) = \s$, and that this $\spinc$ structure is actually a spin structure; see Remark \ref{R:unique}. Our main theorem can now be stated as follows.

\begin{introthm}\label{T:A}
Let $\tau_*: HM^{\red} (\Sigma,\s) \to HM^{\red} (\Sigma,\s)$ be the map in the reduced monopole Floer homology induced by the covering translation $\tau$. Then 
\begin{equation}\label{E:lef}
\Lef(\tau_*)\, =\, n\cdot \lambda (Y)\, +\, \frac 1 8\cdot \sum_{m=1}^{n-1}\;\sign_{m/n}(K)\, -\, h(\Sigma,\s),
\end{equation}
where $\lambda(Y)$ is the Casson invariant of $Y$, $\sign_{m/n} (K)$ are the Tristram--Levine signatures of $K$, and $h(\Sigma,\s)$ is the Fr{\o}yshov invariant.
\end{introthm}

The formula \eqref{E:lef} was first conjectured in \cite{LRS2} but its proof here will use methods completely different from those in \cite{LRS2}: we will first calculate the Seiberg--Witten invariant $\lsw(X)$ for the mapping torus $X$ of $\tau$ using the original definition of $\lsw(X)$ from \cite{MRS}, and then apply the splitting theorem of \cite{LRS1} to derive \eqref{E:lef}. 

\begin{introthm}\label{T:B}
Let $\Sigma$ be a rational homology sphere as above and $X$ the mapping torus of the covering translation $\tau: \Sigma \to \Sigma$ with the standard orientation and homology orientation. Then 
\begin{equation}\label{E:lsw}
-\lsw (X)\, =\, n\cdot \lambda (Y)\, +\, \frac 1 8\cdot \sum_{m=1}^{n-1}\;\sign_{m/n}(K).
\end{equation}
\end{introthm}

We should mention that formula \eqref{E:lsw} was independently proved by Langte Ma \cite{ma:surgery} using a different set of techniques. His result actually holds without the assumption that $\Sigma$ is a rational homology sphere.

%%%%%%%%%%%%%%%%%%%%%%%%%%%%%%%%%%%%%%%%%%%%%%%

\subsection{Motivation} As in our previous work, the research in this paper is motivated by a conjectural relationship between two gauge theoretic invariants of homology $S^{1}\times S^{3}$s, one from the Donaldson theory and the other from the Seiberg--Witten theory. In this section, we discuss our motivation and provide some further implications.

%%%%%%%%%%%%%%%%%%%%%%%%%%%%%%%%%%%%%%%%%%%%%%%

\subsubsection{The Witten--style conjecture for finite-order mapping tori} 
Recall that classical gauge theoretic invariants of a smooth closed oriented 4-manifold $X$ (the Seiberg-Witten invariants and the Donaldson polynomial invariants) are only defined when $X$ satisfies the condition $b^{+}_{2}(X)\geq 1$. However, with suitable modifications, some of these invariants can be defined for other 4-manifolds. In particular, let $X$ be a smooth closed oriented 4-manifold such that
\begin{equation}\label{E:zz}
H_*(X;\Z) = H_*(S^1\times S^3;\Z)\quad\text{and}\quad H_*(\tilde X;\Q) = H_*(S^3;\Q),
\end{equation}
where $\tilde X$ is the universal abelian cover of $X$. Then there are two gauge theoretic invariants for $X$. The first one is the invariant $\lsw(X)$ we mentioned earlier. It was defined in \cite{MRS} by counting Seiberg--Witten monopoles on $X$ and modifying this count by an index-theoretic correction term. The second invariant, called $\lfo(X)$, was defined by Furuta and Ohta \cite{FO} as one quarter times the degree zero Donaldson polynomial of $X$. Furuta and Ohta used somewhat more restrictive hypotheses on $X$; we extended their definition to all manifolds $X$ satisfying \eqref{E:zz} in our paper \cite{LRS2}.

\begin{introconj}[\cite{MRS, LRS2}]\label{conj:lsw-lfo} For any $X$ satisfying (\ref{E:zz}), the following equality holds \[\lsw(X)=-\lfo(X).\] 
\end{introconj}

\noindent
Note that Conjecture \ref{conj:lsw-lfo} relates a Seiberg-Witten type invariant to a Donaldson type invariant. Therefore, it can be thought of as a Witten--style conjecture \cite{witten:monopole} for homology $S^{1}\times S^{3}$s. 

The main results of this paper are inspired by Conjecture \ref{conj:lsw-lfo} for finite order mapping tori, and we confirm this conjecture for the mapping tori of Theorem \ref{T:B}. More precisely, we have the following theorem.

\begin{introthm}\label{T:D}
 Let $\Sigma$ be a rational homology sphere which is a cyclic branched cover of an integral homology sphere, with branch set a knot. Let $X$ be the mapping torus of the covering transformation $\tau:\Sigma\rightarrow \Sigma$. Then $X$ satisfies (\ref{E:zz}) and we have
\[
\lsw (X) = - \lfo (X).
\]
\end{introthm}

\noindent
Note that, in their previous work \cite{RS:jktr}, the second and third authors verified Conjecture \ref{conj:lsw-lfo} for the mapping tori of all finite order diffeomorphisms $\tau: \Sigma \to \Sigma$ of integral homology spheres that generate a free group action on $\Sigma$.

\begin{introcor}\label{C:E}
Let $\Sigma$ be an integral homology sphere and $\tau: \Sigma \to \Sigma$ an orientation preserving diffeomorphism of order $n$ generating a semi-free action of $\Z/n$ on $\Sigma$. If $X$ is the mapping torus of $\tau$ then $\lsw(X)=-\lfo(X)$.
\end{introcor}

%%%%%%%%%%%%%%%%%%%%%%%%%%%%%%%%%%%%%%%%%%%

\subsubsection{Normalized Lefschetz number in monopole and instanton Floer homology} Theorem~\ref{T:D} has an intriguing Floer theoretic interpretation in terms of an expected comparison of the reduced monopole homology groups $\hmred_*(Y)$ of Kronheimer and Mrowka ~\cite{kronheimer-mrowka:monopole} with the reduced instanton homology groups $\hfhat_*(Y)$ of Fr{\o}yshov ~\cite{froyshov:equivariant}.  As an extension of Witten's conjecture~\cite{witten:monopole} on closed $4$-manifolds, such a comparison should also include the maps between the homology groups in question induced by cobordisms. For the product cobordism between \emph{integral} homology spheres, this manifests itself in the known relationship between the Casson invariant and the Euler characteristic of the reduced homologies normalized by the respective $h$-invariants:
\begin{equation}
 \begin{aligned}
\chi(\hmred(\Sigma)) + h(\Sigma) & =  \lambda(\Sigma),\\
\frac12\, \chi(\hfhat(\Sigma)) - h_D (\Sigma) & = \lambda(\Sigma),
\end{aligned} \label{E:hcasson}
\end{equation}
where $h$ and $h_D$ are respectively the monopole and instanton Fr{\o}yshov invariants. Taking \eqref{E:hcasson} as a model, let us define the normalized Lefschetz numbers of a homology cobordism $W$ from an integral homology sphere $\Sigma$ to itself to be
\begin{align*}
& \Lef(W_*:\hmred(\Sigma) \to \hmred(\Sigma)) + h(\Sigma) \quad \text{and}\\
& \frac12\,\Lef(W_*: \hfhat(\Sigma) \to \hfhat(\Sigma)) - h_D(\Sigma)
\end{align*} 
in the monopole and instanton cases, respectively. In the special case of an orientation preserving self-diffeomorphism $\tau: \Sigma \to \Sigma$, the normalized Lefschetz numbers of its mapping cylinder is referred to as the normalized Lefschetz numbers of $\tau$. 

\begin{introcor}\label{C:lefschetz}
If $\Sigma$ is an integral homology sphere and $\tau: \Sigma \to \Sigma$ an orientation preserving diffeomorphism of order $n$ generating a semi-free action of $\Z/n$ on $\Sigma$, then the normalized monopole and instanton Lefschetz numbers of $\tau$ agree.
\end{introcor}

\noindent
This can be seen as follows. Let $X$ be the homology $S^1 \times S^3$ obtained by gluing up the two boundary components of $W$ via the identity map. Then the normalized monopole Lefschetz number of $W$ equals $-\lsw(X)$ by the splitting formula \cite{LRS1}, and the normalized instanton Lefschetz number of $W$ equals $\lfo(X)$ by the splitting formula of Anvari \cite{anvari:splitting}. Comparing this with Corollary~\ref{C:E} completes the proof.

Note that for any 3-manifold,  it is conjectured \cite{kronheimer-mrowka:suture} that the sutured versions of monopole Floer homology and instanton Floer homology coincide. However, relations between all other versions remain mysterious. It is possible that Corollary \ref{C:lefschetz} is a shadow of some deeper relation between these two theories.

%%%%%%%%%%%%%%%%%%%%%%%%%%%%%%%%%%%%%%%%%%%%%%%%%

\subsection{Applications}\label{introapps}
%%When combined with the results of \cite{RS:jktr}, Theorem \ref{T:A} completes the proof of the conjecture that $\lfo(X) = -\lsw(X)$ for finite-order mapping tori. In addition, 
Theorem~\ref{T:A} can be used to study knots in $S^{3}$ and smooth concordance between them. Given a knot $K\subset S^3$ and an integer $n>1$, we let $\Sigma_{n}(K)$ be the $n$-fold cyclic branched cover of $K$. Denote by $L_{n}(K)$ the Lefschetz number of the map $\tau_{*}:HM^{\red} (\Sigma_{n}(K)) \to HM^{\red} (\Sigma_{n}(K))$ induced by the covering transformation $\tau$. Then we have the following corollary of Theorem \ref{T:A}, which generalizes \cite[Corollary F]{LRS2}.

\begin{introcor}\label{C:concordance invariant}
Let $n=p^{m}$ for $p$ a prime number. Then the integer $L_{n}(K)$ is an additive smooth concordance invariant. 
\end{introcor}

Recall that a closed oriented 3-manifold $\Sigma$ is called an $L$-space if $H_1 (\Sigma; \Q) = 0$ and $HM^{\red} (\Sigma;\Z) = 0$. Recent work of Boileau, Boyer, and Gordon~\cite{boileau-boyer-gordon:branched-qp,boileau-boyer-gordon:branched-definite} has focused attention on the question of which branched covers of knots are $L$-spaces. This is of interest in its own right but also as a test case for the $L$-space conjecture~\cite{boyer-gordon-watson:L} equating the property of a rational homology sphere {\em not} being an $L$-space with the left-orderability of its fundamental group. Using $L_{n}(K)$, we can show that the property of a knot $K \subset S^3$ not having an $L$-space branched cover can sometimes hold for an entire concordance class of $K$.

\begin{introthm}\label{T:Donc}
Let $n=p^m$ for $p$ a prime number. There is a knot $K_n$ such that, for any knot $K$ that is smoothly concordant to $K_n$, its $n$--fold cyclic branched cover $\Sigma_n (K)$ is not an $L$-space. 
\end{introthm}

In a further application in this vein, we give a systematic obstruction of branched covers being an $L$-space in terms of the Jones polynomial. 

\begin{introthm}\label{T:Jones}
 Let $K\subset S^{3}$ be a knot with $det(K)=1$ and $J'_K (-1)\neq 0$, where $J_K(t)$ is the Jones polynomial of $K$. Then, for any $m\geq1$, the $2m$-fold cyclic branched cover $\Sigma_{2m}(K)$ is not an $L$-space.
\end{introthm}

\noindent
There are, of course, plenty of examples of knots $K$ whose cyclic branched covers $\Sigma_n (K)$ are not $L$-spaces for all $n$. The novelty of our result is in that it produces such examples in a systematic way. Note that $\Sigma_{2m} (K)$ in the above theorem may fail to be an $L$-space simply because it is not a rational homology sphere. However, since $\det(K) = 1$, the manifold $\Sigma_{2m}(K)$ is an $m$--fold cyclic branched cover of the integral homology sphere $\Sigma_2 (K)$ and so it is automatically a rational homology sphere whenever $m$ is a prime power. 

Finally, recall the conjectural behavior~\cite[Remark 4.2]{RS:jktr} of the invariant $\lsw$ under orientation reversal. 

\begin{introconj}\label{C:triang}
Let $X$ be a smooth spin rational homology $S^1 \times S^3$ which is oriented and homology oriented, and denote by $-X$ the manifold $X$ with reversed orientation but the same homology orientation. Then $\lsw(-X) = -\lsw(X)$. 
\end{introconj}

\noindent 
Note that proving this conjecture would provide an alternate route to the resolution of the triangulation conjecture~\cite{manolescu:triangulation}. We can verify Conjecture \ref{C:triang} in a special case.

\begin{introthm}\label{C:orient}
Let $X$ be the mapping torus of an orientation preserving diffeomorphism (not necessarily of finite order) of a rational homology sphere. Then, for any choice of spin structure on $X$, we have 
\[
-\lsw(X) = \lsw(-X).
\]
\end{introthm}

\noindent
Theorem \ref{C:orient} will actually follow from our splitting formula \cite{LRS1}. We decided to include Theorem \ref{C:orient} 
here because of its relevance to the main result of this paper, Theorem \ref{T:B}.

%%%%%%%%%%%%%%%%%%%%%%%%%%%%%%%%%%%%%%%%%%%%%%%%%%%%%%%%%%

\subsection{An outline of the proof} 
We will first prove Theorem \ref{T:B} by computing the invariant $\lsw (X)$ directly from its definition \cite{MRS},
\[
\lsw (X) = \#\,\M (X,g,\beta) + w(X,g,\beta),
\]
where $g$ and $\beta$ are generic metric and perturbation, $\M(X,g,\beta)$ the Seiberg--Witten moduli space, and $w(X,g,\beta)$ an index theoretic correction term. In the special case of the mapping torus $X$ at hand, we have an orbifold circle bundle $\pi: X \to Y^o$, where $Y^o$ is the orbifold with the underlying space $Y$, the singular set $K$, and the cone angle $2\pi/n$ along $K$. According to Baldridge \cite{Baldridge}, for the right choice of metrics and perturbations, the moduli space $\M(X,g,\beta)$ splits into a disjoint union of the orbifold Seiberg--Witten moduli spaces corresponding to all possible orbifold $\spinc$ structures on $Y^o$. We relate these orbifold moduli spaces to the Seiberg--Witten moduli space on $Y$ in Section \ref{S:monopoles} using an argument reminiscent of the pillowcase argument of Herald \cite{H} in Donaldson's theory. The correction term $w(X,g,\beta)$ for the mapping torus $X$ is just a combination of the $\eta$--invariants of $Y$, which we calculate in Section \ref{S:dirac} and Section \ref{S:sign} using surgery techniques and the splitting formula of Mazzeo and Melrose \cite{MM}. All of this gives us a formula for $\lsw (X)$ in terms of certain invariants of $Y$. It is converted into \eqref{E:lsw} in Section \ref{S:lsw} using the formula of Lim \cite{Lim1} for the Casson invariant in the Seiberg--Witten theory.

Theorem \ref{T:A} follows easily from Theorem \ref{T:B} using the splitting theorem of \cite{LRS1}. Theorem \ref{T:D} is proved in Section \ref{S:lfo}. It can be viewed as a generalization of \cite{RS2} to rational homology spheres, or as a generalization of \cite[Theorem 7.1]{LRS2} to $n \ge 2$, and it is proved by essentially the same methods. Proofs for all of the applications are contained in Section~\ref{S:applications}.

%%%%%%%%%%%%%%%%%%%%%%%%%%%%%%%%%%%%%%%%%%%%%%%%%

\medskip\noindent
\textbf{Acknowledgments:}\; 
We thank the organizers of the 2018 Conference on Gauge Theory at the University of Regensburg, where the broad outline of this project took shape. We thank Tye Lidman for pointing out the references \cite{can-karakurt:lattice,Eftekhary,rustamov:plumbed} on Heegaard--Floer $L$-spaces.

%%%%%%%%%%%%%%%%%%%%%%%%%%%%%%%%%%%%%%%%%%%%%%%%%

\section{Counting Seiberg--Witten monopoles}\label{S:monopoles}
Let $K \subset Y$ be a knot in an oriented integral homology 3-sphere. Denote by $Y^o$ the orbifold with the underlying space $Y$, the singular set $K$, and the cone angle $2\pi/n$ along $K$. In this section, we study monopoles on $Y^o$ using the Seiberg--Witten theory on manifolds with product ends. For the latter, we follow closely the exposition in Lim \cite{Lim2, Lim1}. 

%%%%%%%%%%%%%%%%%%%%%%%%%%%%%%%%%%%%%%%%%%%%%%%%%

\subsection{Monopoles on product end manifods} \label{section: monopole on product end}
Let $D(K)$ be a tubular neighborhood of $K \subset Y$ and $N = Y - \Int (D(K))$ the knot exterior, which is a compact 3-manifold with boundary a 2-torus $T$. Associate with $N$ the product end  manifold
\[
N^* = N\, \largecup_{\,T}\, ([0,\infty) \times T).
\]
Fix a metric $g(N)$ on $N$ that restricts to a flat metric on the boundary and is a product metric in its collar neighborhood. This metric extends in an obvious fashion to a product end metric on the manifold $N^*$.

%Now we specify the metric on these manifolds and orbifolds. Let $g_{D(K)}, g_{D^{o}(K)}$ be the standard flat metric on $D(K)$ and $D^{o}(K)$. Then $g_{D^{o}(K)}$ lifts to a smooth metric $g_{D(\tilde{K})}$ on $D(\tilde{K})$. Fix a metric $g_{0}$ on $N$ which is product near the boundary of $N$; it extends in an obvious fashion to a metric on $\tilde{N}$ which is cylindrical over the end. Let $T$ be the boundary of $N$. For any $L>0$, by attaching a collar $[0,L]\times T^{2}$ to $T$, we obtain a metric $g_{L}$ on $N$, which lifts to a metric $\tilde{g}_{L}$ on $\tilde{N}$. We write the Riemannian 3-manifold $(Y,g_{L}\cup g_{D(K)})$ as $Y_{L}$ and use the notations $Y^{o}_{L},\Sigma_{L}$ in a similar way. We also let $N^{*}=N\cup_{T}([0,+\infty)\times T^{2})$.

%Let $S$ be the spinor bundle over $N$ and let $\mathcal{M}(T)$ be moduli space: 
%$$
%\{A: \text{ spin-c connection on }S|_{T}\mid A \text{ trivial in normal direction, } F_{A^{t}}=0\}/\{\text{gauge transformation on }T\}
%$$

%%%%%%%%%%%%%%%%%%%%%%%%%%%%%%%%%%%%%%%%%%%%%%%%%%

The manifold $Y$ has a unique spin structure. It restricts to a spin structure on $N$, which in turn extends to a spin structure on $N^*$ with spinor bundle $E$. Let $A$ be a unitary connection in the determinant bundle of $E$, and $\phi$ a spinor on $N^*$. Let us consider the $(\omega,\alpha)$-perturbed Seiberg--Witten equations 
\begin{equation}\label{E:sw1}
F_A + \omega = \tau(\phi), \quad \D_A (\phi) + \alpha\cdot \phi = 0,
\end{equation}
where $\omega$ is a closed 2-form on $N^*$ and $\alpha$ is a 1-form on $N^*$, both with coefficients in $i\mathbb R$ and with compact support in $\Int (N)$. The solutions $(A,\phi)$ of these equations will be called \emph{monopoles}. Denote by $\mathcal{M}_{\alpha,\omega}(N^{*})$ the $L^2$--moduli space of monopoles on $N^{*}$ with respect to the gauge group action. The monopoles in $\mathcal{M}_{\alpha,\omega}(N^{*})$ are known to have asymptotic values at infinity, with a flat connection and the zero spinor (a proof can be derived by crossing $N^{*}$ with $S^1$ as in~\cite{taubes:tori}; compare~\cite[Chapter 4.2]{nicolaescu:swbook}). This gives rise to a map
\[
R: \mathcal{M}_{\alpha, \omega} (N^{*})\longrightarrow \chi(T),
\]
where $\chi(T)$ is the moduli space of flat $U(1)$ connections on $\det(E|_{T})$, modulo gauge transformation on $E|_{T}$. Equivalently, we can view  $\chi(T)$ as the $U(1)$ character variety of $\pi_1 (T)$. One can easily see that $\chi(T)$ is a 2-torus; we will introduce a set of coordinates on $\chi(T)$ as follows.

Choose simple closed curves $m$, $\ell$ on the 2-torus $T$ so that $m$ bounds a disk in $D(K)$ and $\ell$ bounds a Seifert surface in $N$. Let $\D(T)$ be the Dirac operator on the 2-torus $T$ associated to the spin connection. 
 An easy calculation shows that there exists a unique point $[A_0] \in \chi(T)$ for which the coupled Dirac operator $\D_{A_0} (T)$ has non-zero kernel.
Then, for any $[A] \in \chi(T)$, we can write $A - A_{0} = \alpha \in \Omega^{1}(T; i \mathbb R)$ and define
\[
m(A) = -2i \int_{m} \alpha\quad \text{and}\quad \ell(A) = -2i \int_{\ell} \alpha.
\]
The assignment of $(m(A),\ell(A))$ to $[A]$ gives a homeomorphism $\chi (T)\to \mathbb {R}^{2}/(2\mathbb{Z})^{2}$. For $a, b\in \mathbb{R}/2\mathbb{Z}$, denote by $[A_{(a,b)}]$ the point in $\chi(T)$ with coordinates $m (A_{(a,b)}) = a$ and $\ell (A_{(a,b)}) = b$. Note that the restriction of the unique flat $U(1)$ connection on $Y$ to the torus $T$ is the flat connection $A_{(1,1)}$. 

We will treat $[A_{(0,0)}] = [A_0]$ as a singular point of $\chi(T)$. For sufficiently small $\epsilon > 0$, denote by $U_{\epsilon}$ the complement in $\chi(T)$ of the closed $\epsilon$--disk centered at $[A_{(0,0)}]$. Also, for future use, denote by $\S_{a}$ the circle in the torus $\chi(T)$ which consists of the points $[A_{(a,b)}]$ with the fixed $a$ and arbitrary $b$.

The monopoles in $\M_{\alpha,\omega} (N^*)$ having $\phi = 0$ will be called \emph{reducible}, and all other monopoles will be called \emph{irreducible}. Denote by $\M^{\red}_{\alpha,\omega} (N^{*})$ and $\M^{\irr}_{\alpha,\omega} (N^{*})$ the reducible, respectively, irreducible loci of $\M_{\alpha,\omega} (N^{*})$. The restrictions of the map $R$ to $\M^{\irr}_{\alpha,\omega} (N^{*})$ and to $\M^{\red}_{\alpha,\omega} (N^{*})$ will be called, respectively,
\[
R^{\text{irr}}:   \M^{\irr}_{\alpha,\omega} (N^{*}) \to \chi(T)\quad\text{and}\quad
R^{\text{red}}: \M^{\red}_{\alpha,\omega} (N^{*}) \to \chi(T).
\]

The following structure theorem for the moduli space $\M_{\alpha,\omega} (N^*)$ is proved in Lim  \cite[Theorem 1.3]{Lim2} and \cite[Theorem 3]{Lim1}.

\begin{thm} \label{moduli space on N}
%The map $R$ gives a diffeomorphism between $\M^{\red}_{0,0} (N^{*})$ and the subset of $\chi (T)$ consisting of the points $[A_{(a,b)}]$ with $b = -1$. 
For any sufficiently small $\epsilon > 0$ there are arbitrarily small perturbations $\alpha $ and $\omega$ such that the following statements hold:  
\begin{itemize}
\item [(1)] The map $R^{\red}$ is a diffeomorphism onto its image, which is a circle contained in $U_{\epsilon}$
\item [(2)] The closure of $(R^{\irr})^{-1} (U_{\epsilon})$ in $\M_{\alpha,\omega} (N^*)$ is a smooth compact 1-manifold with boundary; its boundary points lie in $\M^{\red}_{\alpha,\omega} (N^*)\,\largecup\,R^{-1}(\partial \overline U_{\epsilon})$
\item [(3)] Any boundary point of the closure of $(R^{\irr})^{-1} (U_{\epsilon})$ that lies in $\M^{\red}_{\alpha,\omega} (N^*)$ has a neighborhood in $\M_{\alpha,\omega}(N^*)$ which is modeled on the zeroes of the map $\mathbb R \times \mathbb R^+ \to \mathbb R$ sending $(t,z)$ to $tz$, with $\mathbb R \times \{0\}$ corresponding to the reducibles
\item [(4)] The map $R$ is smooth on $(R^{\irr})^{-1} (U_{\epsilon})$
\item [(5)] Both $(R^{\irr})^{-1} (U_{\epsilon})$ and $\M^{\red}_{\alpha,\omega} (N^*)$ are canonically oriented by the choice of orientation on the real line $H^1 (N^*; \mathbb R)$.
\end{itemize}
\end{thm}

\noindent
From now on, we will always work with $\epsilon$, $\omega$, and $\alpha$ that satisfy the conditions of Theorem \ref{moduli space on N}. There will be further conditions that will require $\epsilon$, $\omega$, and $\alpha$ to be sufficiently small. These conditions will be summarized in Remark \ref{rmk: parameter}.

\medskip

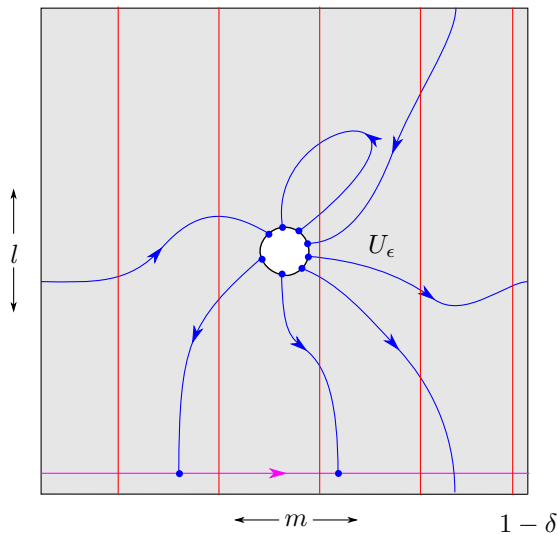
\begin{figure}[ht!]
\labellist
\small\hair 2pt
\pinlabel {$l$} [ ] at 3 204
\pinlabel {$m$} [ ] at 215 3
\pinlabel {$1-\delta$} [ ] at 390 0
\pinlabel {$U_\epsilon$} [ ] at 280 210
\endlabellist
\centering
\includegraphics[scale=0.5]{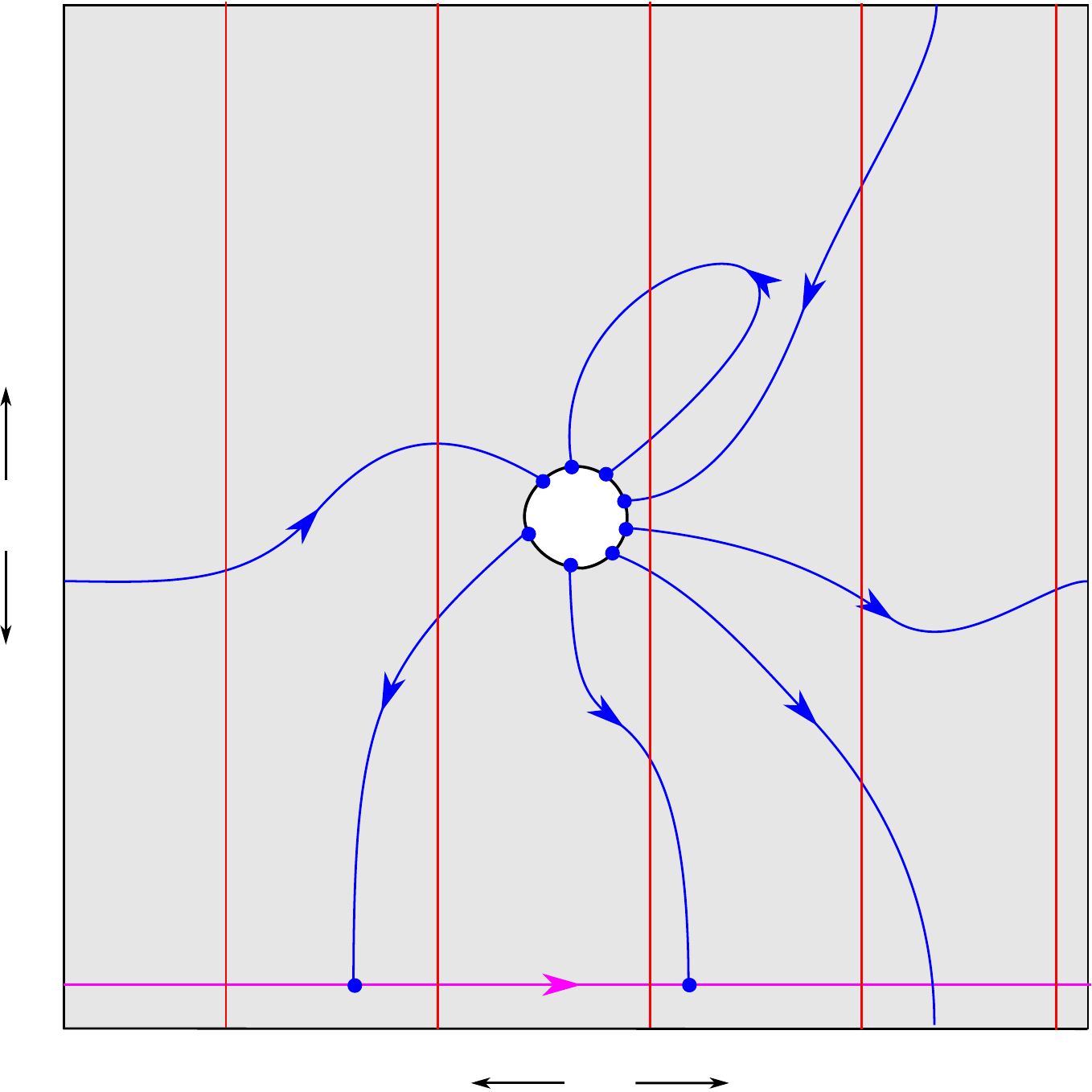}
\caption{The image of $\M_{\alpha,\omega} (N^*)$ in $U_{\epsilon}$}
\label{F:torus}
\end{figure}

One can say more about the map $R^{\red}$, see discussion after Theorem 3 in \cite{Lim2} and Section 7 of \cite{Lim1}. The reducible monopoles are given by the equation $F_A + \omega = 0$. If $\omega = 0$, these are just flat connections on $N^*$, which are mapped to the circle $b = -1$ in the torus $\chi(T)$. The same is true for any $\omega$ with the vanishing cohomology class in $H^2_c (N^*; \mathbb R) = H^2 (N,\p N; \mathbb R) = \mathbb R$. In general, it follows by a direct calculation that 
\[
\ell (A) = -1 +i  \int_F \omega,
\]
where $F$ is a Seifert surface for $\ell \subset N$. In particular, the image of $R^{\red}: \M^{\red}_{\alpha,\omega} (N^*) \to \chi(T)$ is a circle  given by the equation $b = -1 + c(\omega)$, where $c(\omega) \in \mathbb R$ can be made arbitrarily small by choosing small $\omega$. Such a circle is indicated near the bottom of Figure \ref{F:torus}.

%\begin{lem}\label{moduli space on N}
%(1) $R^{\text{red}}$ gives a diffeomorphism between $\mathcal{M}^{\text{red}}(N^{*})$ and the subset $\{[A_{*,-1}]\}$ of $\mathcal{M}(T)$. 

%(2) For $a\in \mathbb{R}/2\mathbb{Z}$, let $([(\theta(N^{*},a),0)]=(R^{\text{red}})^{-1}([A_{a,-1}])$. Then as $a$ goes around the circle $\mathbb{R}/2\mathbb{Z}$, the spectral flow of the coupled Dirac operator $\slashed{D}_{\theta(N^{*},a)}$ (over $N^{*}$) equals zero. (One needs to make sense of the spectral of this operator, see \cite{Lim1} for related discussion.)
%\end{lem}

% For now, we put the following assumption. In general case, it can be achieved by introducing a perturbation.
%\begin{assum} The closure of  
%$(R^{\text{irr}})^{-1}(U_{\epsilon})$ in $\mathcal{M}(N^{*})$ is a 1-manifold with boundary, the boundary point lies in $R^{-1}(\partial U_{\epsilon})\cup \mathcal{M}^{\text{red}}(N^{*})$.
%\end{assum}

%%%%%%%%%%%%%%%%%%%%%%%%%%%%%%%%%%%%%%%%%%%%%%%%%

\subsection{Monopoles on orbifolds} \label{section: monopole orbifold} 
Recall that we denoted by $Y^o$ the orbifold with the underlying space $Y$, the singular set $K$, and the cone angle $2\pi/n$ along $K$, and that we equipped the knot exterior $N=Y^{o}-D^{o}(K)$ with a metric $g(N)$ which restricts to a flat metric on the boundary $T$. We will further assume that, with respect to this flat metric, the meridian $m$ and the longitude $\ell$ are orthogonal geodesics of length $u$ and $v$, respectively. We will equip $Y^o$ with an orbifold metric obtained by gluing $g(N)$ to an orbifold metric $g^{o}_{u,v}$ on $D^{o}(K)$. The latter metric is defined as follows.
\begin{defi}\label{defi: bullet-type metrics}
Let $D^{2}\times S^{1}$ be a solid torus with the polar coordinates $(r,\theta,\xi)$, where $0 \le r \le \pi/2$ and $\theta,\xi\in \mathbb{R}/2\pi$. For any $u, v > 0$, define the metric $g_{u,v}$ on $D^{2}\times S^{1}$ by the formula
$$
\frac{v^{2}}{4\pi^{2}}\;d\xi\otimes d\xi+\frac{u^{2}}{4\pi^{2}}\,\left(dr\otimes dr+ h(r)^{2} d\theta\otimes d\theta\right),     
$$ 
where $h: \left[0,\pi/2\right]\rightarrow \mathbb{R}$ is a smooth function such that
\begin{itemize}
    \item $h(r)=\sin r$\, when\, $0 \le r \le \pi/6$,
    \item $h(r)=1$\, when\, $5\pi/12 \le r \le \pi/2$, and 
    \item $h''(r) < 0$\, when\, $\pi/6 \le r\le 5\pi/12$.
\end{itemize}
We call $g_{u,v}$ the {\it smooth bullet-type metric}. With respect to this metric, the boundary of $D^2$ has length $u$, and the circle factor has length $v$. We define the {\it orbifold bullet-type metric} $g^{o}_{u,v}$ as the quotient metric of $g_{nu,v}$ under the finite-order isometry $\iota$ of $D^{2}\times S^{1}$ taking $(r,\theta,\xi)$ to $(r,\theta+2\pi/n,\xi)$. 
Note that both $g_{u,v}$ and $g^{o}_{u,v}$ are flat near the boundary and have positive scalar curvature $\scal(g_{u,v})$ and $\scal(g^{o}_{u,v})$ elsewhere.
\end{defi}

We will need to perturb the Seiberg--Witten equations near $K$. To do this, we let $\nu$ be the 2-form on $D^{2} \times S^{1}$ obtained by pulling back an $i\mathbb R$--valued 2-form on the unit disk $D^{2}$ which is supported in the region $r\leq \pi/6$ and integrates to $i/2$. We define the orbifold 2-form $\nu^{o}$ to be the quotient of the form $n\cdot \nu$ by the action of the isometry $\iota$.  The following lemma is straightforward.

\begin{lem}\label{lem: curvature bound}
There exists a constant $\delta_{0}>0$ such that the inequalities
\begin{equation}\label{eq: curvature control}
|\delta_{0}\nu(x)|\leq \scal(g_{u,v})(x)\quad \text{and} \quad |\delta_{0}\nu^{o}(y)|\leq \scal(g^{o}_{u,v})(y)
\end{equation}
hold for all $x\in D(K)$ and $y\in D^{o}(K)$.
\end{lem}
\noindent
In what follows, we will choose $\epsilon > 0$ (the radius of the disk around the singularity in $\chi(T^{2})$) smaller than the constant $\delta_{0} > 0$ of Lemma \ref{lem: curvature bound}.

Denote by $g(Y^{o})$ the orbifold metric on $Y^o = N\,\largecup_T\, D^o (K)$ obtained by gluing together the metrics $g(N)$ and $g^{o}_{u,v}$. Recall that the underlying space of $Y^{o}$ is an integral homology 3-sphere. Therefore, the orbifold $\spinc$ structures $\s_{k}$ on $Y^o$ can be canonically parametrized by integers $0\leq k\leq n-1$; see Baldridge \cite[Theorem 7 and Theorem 4]{Baldridge}. Given an orbifold $\spinc$ structure $\s_{k}$, consider the orbifold Seiberg--Witten equations as in Baldridge \cite[Section 2.4]{Baldridge},
\begin{equation}\label{E:sw3}
F_A  + \omega^o = \tau(\phi),\quad \D_A (\phi) = 0,
\end{equation}

\smallskip\noindent
where $\omega^o$ is a closed orbifold 2--form on $Y^o$ with coefficients in $i \mathbb R$. Equations \eqref{E:sw3} give rise to the orbifold Seiberg--Witten moduli space, which will be called $\M (Y^o,\s_k,\, g(Y^o),\omega^o)$. This moduli space can be recovered from $\M_{\alpha,\omega} (N^*)$ by using the right choice of metric and perturbation on $Y^o$. The metric will be the metric $g(Y^o)_L$ on the orbifold
\[
Y^o = N\,\largecup_T\,([-L,L]\times T)\,\largecup_T\, D^o (K)
\]
obtained from $g(Y^o)$ by `neck stretching', and the perturbation 2--form $\omega^o_{\delta} = \omega + \delta\cdot \nu^o$ being the sum of the 2--form $\omega$ supported in $\Int(N)$, and the orbifold 2--form $\delta \cdot \nu^o$ supported in $\Int(D^{o}(K))$, which integrates to $i \delta/2$ on each meridional disk.

\begin{pro}\label{gluing2} 
For a generic $\delta\in (\epsilon,\delta_{0})$, the map $R^{\irr}$ is transversal to $\S_{2k/n + 1 - \delta}$ for all $k$. Moreover, for all sufficiently large $L  > 0$, the zero-dimensional manifolds 
$$
\M^{\irr} (Y^o,\s_k,\, g(Y^o)_L,\,\omega^o_\delta - d\alpha)\quad \text{and}\quad  (R^{\irr})^{-1} (\S_{2k/n + 1 - \delta})
$$
are orientation preserving diffeomorphic.
\end{pro}

\begin{proof}
The transversality assertion follows from Theorem \ref{moduli space on N} and Sard's Theorem. The identification between $\M^{\irr} (Y^o,\s_k,\, g(Y^o)_L,\,\omega^o_\delta - d\alpha)$ and  $(R^{\irr})^{-1} (\S_{2k/n + 1 - \delta})$
is an orbifold version of \cite[Theorem 1.4]{Lim2} and \cite[Theorem 4]{Lim1}, which is essentially a gluing argument along a torus, hence the presence of orbifold points makes no difference.  The key ingredient is the inequality \eqref{eq: curvature control}, which implies that the $\delta\cdot \nu^o$--perturbed Seiberg--Witten equations on the cylindrical-end orbifold $((-\infty,0]\times T)\cup_{T}D^{o}(K)$ have no irreducible solutions with finite energy.
\end{proof}

%\begin{rmk}\label{R:delta}
%For even $n$, we are forced to work with $\delta > 0$ to ensure that the circle $\S_{-\delta}$ fits into $U_{\epsilon}$, at least for sufficiently small $\epsilon > 0$. % Note that Proposition \ref{gluing2} reduces to Proposition \ref{gluing1} for the $\spinc$ structure $\s_0$. 
%\end{rmk}

 An argument similar to that in \cite[page 637]{Lim1} then shows that the oriented count of points in $(R^{\irr})^{-1} (\S_{2k/n+1-\delta})$ is given by the intersection number of $R (\M^{\irr}_{\alpha,\omega} (N^*))$ with the circle $\S_{2k/n+1-\delta}$ in the torus $\chi(T)$, where the torus $\chi(T)$ is oriented by $\p/\p a \wedge \p/\p b$, and the circle $\S_{2k/n+1-\delta}$ is oriented by $\p/\p b$ (again, the presence of orbifold points makes no difference). By combining this observation with Proposition \ref{gluing2}, we obtain the following result. 

\begin{cor}\label{C:orbifold}
For a generic $\delta\in (\epsilon,\delta_{0})$ and for all $L>0$ sufficiently large, the oriented count of points in the moduli space $\M^{\irr} (Y^o,\s_k,\,g(Y^o)_L,\omega^o_{\delta} - d\alpha)$ equals the intersection number of $R (\M^{\irr}_{\alpha,\omega} (N^*))$ with the circle $\S_{2k/n+1-\delta}$ in the torus $\chi(T)$.
\end{cor}

%%%%%%%%%%%%%%%%%%%%%%%%%%%%%%%%%%%%%%%%%%%%%%%%%%%

\subsection{Monopoles on $Y$}
Now we study the perturbed Seiberg--Witten equations on the homology sphere $Y$. Let us consider the splitting $Y = N\,\largecup_T\, D(K)$ with a metric $g(Y)$ which restricts to the metric $g(N)$ on $N$ and a smooth bullet-type metric $g_{u,v}$ on $D(K) = S^1 \times D^2$ (see Definition \ref{defi: bullet-type metrics}). For any constant $L > 0$, equip the manifold 
\[
Y = N\,\largecup_T\,([-L,L]\times T)\;\largecup_T\, D(K)
\]
with the metric $g(Y)_L$ obtained from $g(Y)$ by `neck stretching'. In addition, let $\omega_{\delta} = \omega + \delta\cdot \nu$ be the sum of the form $\omega$ supported on $N$  and the 2--form $\delta\cdot \nu$ supported on $D(K)$ (as defined in Section \ref{section: monopole orbifold}), which integrates to $i \delta/2$ on each meridional disk.
Then the $(\omega_{\delta},\alpha)$--perturbed Seiberg-Witten equations, as studied by Lim \cite{Lim1}, are of the form
\begin{equation}\label{E:sw2}
F_A + \omega_{\delta} = \tau(\phi), \quad \D_A (\phi) + \alpha\cdot \phi = 0.
\end{equation}
The corresponding moduli space $\M_{\alpha,\omega_{\delta}}(Y,g(Y))$  contains exactly one reducible solution.  We wish to compare $\M_{\alpha,\omega_{\delta}}(Y,g(Y))$ with the moduli space $\M(Y,g(Y),\beta)$ (which was used in \cite{MRS} to define $\lsw$) given by the equations
\begin{equation}\label{E:sw4}
F_A = \tau(\phi) + d\beta,\quad \D_A (\phi) = 0.
\end{equation}
To this end, observe that any closed 2-form $\omega_{\delta}$ is exact on the homology sphere $Y$, hence can be written in the form $\omega_{\delta} = d\gamma_{\delta}$, where $\gamma_{\delta}$ is a 1-form with coefficients in $i\mathbb R$. 

\begin{pro}\label{P:diff1}
The choice of $\beta_{\delta} = \alpha - \gamma_{\delta}$ establishes an orientation preserving diffeomorphism between the moduli spaces $\M_{\alpha,\omega_{\delta}} (Y,g(Y))$ and $\M(Y,g(Y),\beta_{\delta})$.
\end{pro}

\begin{proof}
The change of variables $B = A + \alpha$ in the Seiberg--Witten equations \eqref{E:sw2} results in the equations $F_B - d\alpha + d \gamma_{\delta} = \tau(\phi)$ and $\D_B (\phi) = 0$, which match the equations \eqref{E:sw4} defining $\M(Y,g(Y),\beta_{\delta})$ once we set $\beta_{\delta} = \alpha - \gamma_{\delta}$. The result now follows by comparing the orientation conventions for the two moduli spaces.
\end{proof}

 The following result is a special case of Corollary \ref{C:orbifold} when $n=1$.
 
 \begin{cor}\label{C:manifold}
For a generic $\delta\in (\epsilon,\delta_{0})$, the map $R^{\irr}$ is transversal to $\S_{1 - \delta}$. Moreover, for all $L>0$ sufficiently large, the oriented count of points in the moduli space $\M^{\irr} (Y,g(Y)_L,\beta_{\delta})$ equals the intersection number of $R (\M^{\irr}_{\alpha,\omega} (N^*))$ with the circle $\S_{1-\delta}$ in the torus $\chi(T)$. 
\end{cor}

\subsection{The spectral flow formula}
According to Part (1) of Theorem \ref{moduli space on N}, the moduli space $\M^{\red}_{\alpha,\omega}(N^*)$ is a circle, for any sufficiently small generic perturbations $\alpha$ and $\omega$. This circle admits a parametrization by $\theta (N^*,a) = (R^{\red})^{-1}(\S_a)$, thereby giving rise to a family of twisted Dirac operators $D_{\theta(N^*,a)}$ on the product end manifold $N^*$. The $L^2$ completions of these operators are self-adjoint, which leads to a well-defined notion of spectral flow; see for instance Cappell, Lee, and Miller \cite{CLM}. The points  where the operators $D_{\theta(N^*,a)}$ have non-zero kernels are precisely the points in $\M^{\red}_{\alpha,\omega} (N^*)$ which serve as the boundary points of the closure of $(R^{\irr})^{-1} (U_{\epsilon})$. It then follows from Part (3) of Theorem \ref{moduli space on N} that, for any choice of perturbations $\alpha$ and $\omega$ as in that theorem, the spectral flow is transverse, that is, the spectral curves of the family $D_{\theta(N^*,a)}$ intersect the $a$--axis transversely with multiplicity one. 

\begin{lem}\label{L:sf}
For sufficiently small $\alpha$ and $\omega$, the spectral flow of the family of twisted Dirac operator $\D_{\theta(N^*,a)}$ around the circle is equal to zero. 
\end{lem}

\begin{proof}
Let $f: N^* \to S^1$ be an arbitrary smooth function which induces an isomorphism $f_*: H_1 (X) \to \Z$ and consider the family $\D_a = \D + ia\, df$ of twisted Dirac operators on $N^*$. The unperturbed spin Dirac operator $\D$ is quaternionic linear, hence the $j$--conjugate of $\D_a$ is $\D_{-a}$, which makes the picture of the spectral curves of $\D_a$ symmetric with respect to the involution sending $a$ to $-a$. This ensures that the spectral flow of the family $\D_a$, which was defined in \cite[Section 7]{APS:III} as the intersection number of the spectral curves with a small vertical shift of the $a$--axis, vanishes. The family $\D_{\theta(N^*,a)}$ is a deformation of the family $\D_a$ hence its spectral flow vanishes as well, as long as the perturbations $\alpha$ and $\omega$ are sufficiently small. 
\end{proof}

In what follows, we will choose a generic $\delta > 0$ such that following condition is satisfied:
\begin{equation}\label{eq: Dirac invertible}
  \text{the operators}\;\, \D_{\theta(N^*,a)}\;\, \text{are invertible at}\;  a = 2k/n - 1 - \delta\; \text{for all $k$}. 
\end{equation}

\begin{pro}\label{orbifold versus manifold} 
For all sufficiently large $L > 0$, sufficiently small generic perturbations $\alpha$ and $\omega$, and generic $\delta > 0$, one has the relation
\[
\begin{split}
\#\M^{\irr}(Y^o, \s_k,g(Y^o)_L,\omega^o_{\delta} - d\alpha)\, & =\, \#\M^{\irr}(Y, g(Y)_L,\beta_{\delta})\\
& \quad -\, \SF (\D_{\theta(N^*, 1-\delta)}, \D_{\theta(N^*, 2k/n+1-\delta)}).
\end{split}
\] 
The spectral flow in this formula is calculated along any path in the circle $\theta(N^*,a)$ that leads from $\theta(N^*, 1-\delta)$ to $\theta(N^*, 2k/n +1-\delta)$; according to Lemma \ref{L:sf}, this is a well defined quantity.
\end{pro}

\begin{proof} 
According to Corollary \ref{C:orbifold} and Corollary \ref{C:manifold}, we have the following identities
\begin{gather*}
\#\M^{\irr}(Y^o, \s_k, g(Y^o)_L, \omega^o_{\delta} - d\alpha) = \#(R(\M^{\irr}_{\alpha,\omega} (N^*))\,\largecap\,\S_{2k/n+1-\delta})\quad\text{and}\\
\#\M^{\irr}(Y, g(Y)_L,\beta_{\delta})=\#(R(\M^{\irr}_{\alpha,\omega} (N^*))\,\largecap\,\S_{1-\delta}).
\end{gather*}
According to Theorem \ref{moduli space on N}, the one-dimensional moduli space $\M^{\irr}_{\alpha,\omega} (N^*)$ provides an oriented cobordism between the points in $\M^{\irr}_{\alpha,\omega} (N^*)\,\largecap\,R^{-1}(\S_{1-\delta})$, the points in $\M^{\irr}_{\alpha,\omega} (N^*)\,\largecap\,R^{-1}(\S_{2k/n-1-\delta})$, and the reducible boundary points in $\M^{\red}_{\alpha,\omega} (N^*)$ that fit into the interval between $\theta (N^*,1-\delta)$ and $\theta(N^*,2k/n+1-\delta)$; of the two intervals in the circle $\M^{\red}_{\alpha,\omega} (N^*)$ having the same endpoints, we chose the one that does not contain $\theta(N^*,0)$ so as to stay within $U_{\epsilon}$. The reducible boundary points in $\M^{\red}_{\alpha,\omega} (N^*)$ correspond to the points where the Dirac operators $D_{\theta(N^*,a)}$ have non-zero kernels hence these are precisely the points that contribute to the spectral flow in the statement of the proposition. That they contribute with the right sign follows from the description of the signs in \cite[Section 11.3]{Lim2} and \cite[page 635]{Lim1}: when moving in the direction of the orientation of $\M^{\red}_{\alpha,\omega} (N^*)$, the contribution to the spectral flow at $\theta(N^*,a)$ is positive  if and only if the orientation of $\M^{\irr}_{\alpha,\omega} (N^*)$ is into $\theta(N^*,a)$. 
\end{proof}

%%%%%%%%%%%%%%%%%%%%%%%%%%%%%%%%%%%%%%%%%%%%%%%%%%

\section{Eta-invariant of the Dirac operator}\label{S:dirac}
Let $X$ be the mapping torus of the covering translation $\tau: \Sigma \to \Sigma$ with the standard orientation and homology orientation. A choice of spin structure on $X$ induces a spin structure on $\Sigma$. This induced spin structure is invariant under $\tau$ hence it is the same spin structure that lifts the unique spin structure on the integral homology sphere $Y$. We will have these particular spin structures in mind when talking about the spin manifolds $\Sigma$ and $Y$. In this section, we will compare the $\eta$--invariants of the (twisted) spin Dirac operators on $\Sigma$ and $Y$. The idea for this comparison comes from Lim \cite{Lim1}. 

For the knot $K \subset Y$, consider its preimage $\widetilde K \subset \Sigma$ under the branched cover projection $\Sigma \to Y$. The knot $\widetilde K$ can also be viewed as the fixed point set of the covering translation $\tau: \Sigma \to \Sigma$. Write
\begin{equation}\label{E:dec1}
Y = N\, \largecup\, D(K)\quad\text{and}\quad \Sigma = N_n\, \largecup\, \widetilde D(K),
\end{equation}
where $N_n \to N$ is a regular $n$--fold cover. Let $Y (0)$ be the manifold obtained from $Y$ by 0-surgery on the knot $K$, and similarly $\Sigma(0)$ the manifold obtained from $\Sigma$ by 0-surgery on $\widetilde K$. Write
\begin{equation}\label{E:dec2}
Y(0) = N\, \largecup\, D(0)\quad\text{and}\quad \Sigma(0) = N_n\, \largecup\, \widetilde D(0),
\end{equation}
where $D(0)$ is a solid torus $D^{2}\times S^{1}$ glued to $N$ along its torus boundary by sending its longitude and meridian to the curve $m$ and $l$ in $T=\partial N$ respectively (see Section \ref{section: monopole on product end}), and similarly for $\widetilde D(0)$. The regular $n$--fold cover $N_n \to N$ extends to a regular $n$--fold cover $\Sigma(0) \to Y(0)$. The aforementioned spin structures on $Y$ and $\Sigma$ give rise to spin structures on $Y(0)$ and $\Sigma(0)$. Denote by $g(Y(0))$ the metric on the manifold $Y_{0}$ that equals $g_{N}$ on $N$ and the bullet-type metric $g_{v,u}$ on $D(0)$. We use the symbol $g(\Sigma)$ and $g(\Sigma(0))$ to denote the pull back of the metrics $g(Y^{o})$ and $g(Y(0))$. For any positive constant $L$, denote by $g(*)_L$ the metrics on these manifolds obtained from $g(*)$ by neck stretching as in Section \ref{S:monopoles}; these metrics will sometimes be suppressed in our notations.

Let $\theta (Y)$ be the unique reducible monopole in $\M_{\alpha,\omega_{\delta}} (Y, g(Y)_L)$ for the choices of $\alpha$, $\omega_{\delta}$ and $L$ as Corollary \ref{C:manifold}; this monopole is the unique extension to $Y$ of the reducible monopole $\theta (N^*,1-\delta)$ on $N^*$ whose limiting value is the flat connection $A_{(1-\delta,-1+c(\omega))}$. Denote by $\theta (\Sigma)$ the lift of $\theta (Y)$ to $\Sigma$. 

Note that each reducible monopole $\theta (N^*,a)$ on $N^*$ admits a unique extension to a reducible monopole on $Y(0)$, which will be called $\theta(Y(0),a)$. This follows by applying Theorem \ref{moduli space on N} to the manifold $D(0)$ in place of $N$, and matching the limiting values in $\chi(T)$ by choosing an appropriate perturbation on $D(0)$. This perturbation can be made arbitrary small by choosing small $\alpha$ and $\omega$. A similar construction applied to $\Sigma (0)$ in place of $Y(0)$ gives rise to the monopole $\theta(\Sigma(0),a)$ on $\Sigma(0)$. The monopoles $\theta(Y(0),1-\delta)$ and $\theta(\Sigma(0),1-\delta)$ will be referred to as simply $\theta(Y(0))$ and $\theta(\Sigma(0))$.

%We will consider the following spin-c connections:
%\begin{itemize}
%\item $\theta(Y_{L})$: flat connection on $Y_{L}$;
%\item $\theta(\Sigma_L)$: flat connection on $\Sigma_L$;
%\item $\theta(Y_{L}(0),1)$: flat connection on $Y_{L}(0)$ such that $\theta(Y_{L}(0),1)|_{T}=A_{(-1,-1)}$;
%\item $\theta(Y_{L}(0),0)$: flat connection on $Y_{L}(0)$ such that $\theta(Y_{L}(0),0)|_{T}=A_{(0,-1)}$;
%\item $\theta(\Sigma_{L}(0),1)$: lift of $\theta(Y_{L}(0),1)$ to $\Sigma_{L}(0)$ under the covering map.
%\end{itemize}

\begin{lem}\label{dirac-eta surgery}
For any generic $\delta > 0$ that satisfied (\ref{eq: Dirac invertible}), there exist constants $C_1$ and $C_2$ which are independent of $Y$, $\Sigma$, $K$, and $L$ (but may depend on $\delta$) such that, for all sufficiently large $L > 0$, the following property holds
\begin{gather}
\eta(\D_{\theta(Y)})\, =\, \eta(\D_{\theta(Y(0))}) + C_1 + o(1), \label{eta compare 1} \\
\eta(\D_{\theta(\Sigma)})\, =\, \eta(\D_{\theta(\Sigma(0))}) + C_2 + o(1), \label{eta compare 2}
\end{gather}
where $o(1)$ as usual denotes a quantity that limits to zero as $L$ goes to infinity.
\end{lem}

\begin{proof}  
The restrictions of the monopoles $\theta(Y)$ and $\theta(Y(0))$ to the tubular neighborhoods $D(K)$ and $D(0)$ give rise to the monopoles on the product end manifolds 
\[
D^*(K) = ((-\infty,0]\times T)\largecup_{T} D(K)\quad\text{and}\quad D^*(0) = ((-\infty,0]\times T)\largecup _{T}D(0)
\] 
which will be called, respectively, $\theta(D^*(K))$ and $\theta(D^*(0))$. Since both manifolds $D^*(K)$ and $D^*(0)$ have non-negative scalar curvature, the Dirac operators $\D_{\theta(D^*(K))}$ and $\D_{\theta(D^*(0))}$ are invertible. Combined with the invertibility of $\D_{\theta(N^*,1-\delta)}$, this implies that, for all sufficiently large $L > 0$, the operators $\D_{\theta(Y)}$ and $\D_{\theta(Y(0))}$ are invertible. Moreover, it follows from Mazzeo and Melrose \cite{MM} that
\begin{gather*}
\eta(\D_{\theta(Y)})\,  =\, \eta(\D_{\theta(N^*,1-\delta)}) + \eta(\D_{\theta(D^*(K))}) + o(1)\quad\text{and}\; \\
\eta(\D_{\theta(Y(0))})\,  =\, \eta(\D_{\theta(N^*,1-\delta)}) + \eta(\D_{\theta(D^*(0))}) + o(1).\qquad
\end{gather*}
Subtracting these two formulas, we obtain formula (\ref{eta compare 1}) with the constant $C_1 = \eta(\D_{\theta(D^*(K))}) -\eta(\D_{\theta(D^*(0))})$, which is a linear combination of the $\eta$--invariants of standard solid tori and hence is independent of $Y$, $K$, and $L$. Formula (\ref{eta compare 2}) is proved similarly.
\end{proof}

\begin{lem} 
One has the following identity 
\[
\eta(\D_{\theta(\Sigma(0))})\; =\; \sum_{k=0}^{n-1}\; \eta(\D_{\theta(Y(0),\,2k/n + 1 -\delta)}).
\]
\end{lem}

\begin{proof} 
Since $\theta(\Sigma(0))$ is the lift of $\theta(Y(0))$ under the regular $n$--fold covering map $\Sigma(0) \to Y(0)$, the covering translation $\tau: \Sigma(0) \to \Sigma(0)$ induces an action on the spinors, splitting each operator $\D_{\theta(\Sigma(0))}$ into a direct sum of the operators $\D^{\alpha}_{\theta(\Sigma(0))}$ on its eigenspaces. This in turn leads to the identity
\[
\eta(\D_{\theta(\Sigma(0))})\; =\; \sum_{\alpha}\; \eta (\D^{\alpha}_{\theta(\Sigma(0))}).
\]

\smallskip\noindent
One can easily check that the operators $\D^{\alpha}_{\theta(\Sigma(0))}$ in this formula are precisely the operators $\D_{\theta(Y(0),\,2k/n + 1 -\delta)}$, which completes the proof.
\end{proof}

\begin{lem}\label{dirac-eta spectral flow} 
For all sufficiently large $L > 0$ and all $k$, one has 
\[
\eta(\D_{\theta(Y(0),2k/n+1-\delta)})\, =\, \eta(\D_{\theta(Y(0))})\, +\, 2 \SF(\D_{\theta(N^*,\,1-\delta)}, \D_{\theta(N^*,\, 2k/n+1-\delta)}).
\]
\end{lem}

\begin{proof}
According to \cite{APS:III}, the difference $\eta(\D_{\theta(Y(0),\, 2k/n + 1 -\delta)}) - \eta(\D_{\theta(Y(0))})$ equals twice the spectral flow $\SF (\D_{\theta(Y(0))}, \D_{\theta(Y(0),\, 2k/n+1-\delta)})$. According to \cite{CLM}, for all sufficiently large $L > 0$, this spectral flow is the sum of the spectral flow over $D^*(0)$ and the spectral flow over $N^*$. The former vanishes because the metric on $D^*(0)$ has non-negative scalar curvature and the perturbation is small, and the latter equals the spectral flow $\SF (\D_{\theta(N^*,\,1-\delta)}, \D_{\theta(N^*,\,2k/n+1-\delta)})$.
\end{proof}

%By combining the three above lemmas together, we arrive at the main result of this section.

\begin{cor}\label{dirac eta compare 3}
For any $\delta$ that satisfies (\ref{eq: Dirac invertible}), there exists a constant $C'$ which, for all sufficiently large $L > 0$, is independent of $Y$, $\Sigma$, $K$, and $L$, and has the property that
\[
\eta(\D_{\theta(\Sigma)})\, =\, n\cdot \eta(\D_{\theta(Y)})\, +\, 2\, \sum_{k=0}^{n-1}\; \SF(\D_{\theta(N^*,\, 1-\delta)}, \D_{\theta(N^*,\,2k/n+1-\delta)}) + C' + o(1).
\]
\end{cor}

\begin{rmk}\label{rmk: parameter}
So far, we have made choices of the following parameters:
\begin{itemize}
\item $\epsilon > 0:$ the radius of the disk around the singular point $[A_{(0,0)}]\in \chi(T)$,
\item $\alpha,\omega:$ differential forms on $N^*$ used to perturb the Seiberg--Witten equations, and
\item $\delta>0:$ a positive number whose products with the fixed 2-forms $\nu$ and $\nu^{o}$ serve as the perturbations on $D(K)$ and $D^{o}(K)$.
\end{itemize}
(Other parameters, such as the bullet type metric $g_{u,v}$ and the 2-form $\nu$ on $D(K)$, have been fixed since the moment they were defined).
We have imposed several constraints on these parameters. To clarify that these constrains do not contradict each other, we summarize them here: First, we choose $\epsilon > 0$ to be small enough so that Theorem \ref{moduli space on N} applies. The parameter $\epsilon > 0$ also needs to be smaller than the constant $\delta_{0} > 0$ of Lemma \ref{lem: curvature bound}. Then we choose generic $\alpha$ and $\omega$ that meet the requirements of Theorem \ref{moduli space on N}. The perturbations $\alpha$ and $\omega$ also need to be small enough so that Lemma \ref{L:sf}, Proposition \ref{orbifold versus manifold}, and all results of Section \ref{S:dirac} hold. Finally, we choose a generic $\delta\in (\epsilon,\delta_{0})$ to satisfy the condition (\ref{eq: Dirac invertible}) and the condition of Corollary \ref{C:orbifold}. We will fix all these choices from now on and will not discuss them further.
\end{rmk}

%%%%%%%%%%%%%%%%%%%%%%%%%%%%%%%%%%%%%%%%%%%%%%%%%%%

\section{Eta-invariant of the signature operator}\label{S:sign}
In this section, we will analyze the eta-invariants of the odd signature operators on the manifolds $M = Y$, $\Sigma$, $Y(0)$, and $\Sigma(0)$. As before, we fix the long neck metrics $g(M)_{L}$ on these manifolds and use $\eta_{\sign}(M)$ to denote the corresponding eta-invariants. Since $\eta_{\sign} (M)$ is closely related to a topological invariant (the signature), its analysis is much easier than that of the eta-invariant of the Dirac operator; cf. \cite{meyerhoff-ruberman:II}.

%The trick is to use surgery to turn the branched cover into unbranched one. To do this, we let $D(0)$ be a solid torus bounded by $T$, with $l$ bounding a compressing disk. We fix a flat metric $g_{D(0)}$ on $D(0)$ that extends the flat metric $g_{D(K)}|_{T}$ on $T$. Let $\tilde{D}(0)$ be the (unbranched) double cover of $D(0)$ and let  $g_{\tilde{D}(0)}$ be the lift of $g_{D(0)}$. Then $\Sigma(0):=\tilde{N}\largecup_{\tilde{T}} \tilde{D}(0)$ is a double cover of $Y(0):=N\largecup_{T} D(0)$. Under this covering map, the metric $g_{L}\largecup g_{D(0)}$ lifts to the metric $\tilde{g}_{L}\largecup g_{\tilde{D}(0)}$. We use $Y_{L}(0)$ and $\Sigma_{L}(0)$ to denote these Riemannian manifolds.

\begin{lem} One has the following identities
\[
\eta_{\sign}(\Sigma) =\,\eta_{\sign}(\Sigma(0))\quad {\rm and} \quad
\eta_{\sign}(Y)=\eta_{\sign}(Y(0)).
\]
\end{lem}
\begin{proof}
This will follow from an excision argument for $\eta_{\sign}(M)$. We will focus on the second equality since the first one is similar. To keep better track of notation, we will denote by 
$$
f_{1}:\partial D(K)\rightarrow T=\partial N\quad \text{and}\quad f_{0}:\partial D(0)\rightarrow T=\partial N
$$
the gluing maps for $Y$ and $Y(0)$, and denote by $g(T)$ the restriction of the metric $g(N)$ to $T$. We form a Riemannian 4-manifold with corners,
$$
W=([0,3]\times [-L,L]\times T)\cup ([0,1]\times D(K)) \cup ([2,3]\times D(0)),
$$
where $[0,1]\times D(K)$ is glued to $[0,3]\times [-L,L]\times T$ via the map
$$
\text{id}_{[0,1]}\times f_{1}: [0,1]\times \partial D(K)\rightarrow [0,1]\times \{L\}\times T,
$$
and $[2,3]\times D(0)$ is glued to $[0,3]\times [-L,L]\times T$ via the map 
$$
\text{id}_{[2,3]}\times f_{0}: [2,3]\times \partial D(K)\rightarrow [2,3]\times \{L\}\times T.
$$
A schematic picture of this region is depicted in Figure \ref{F:eta}.
 
\bigskip
 
\begin{figure}[ht!]
\labellist
\scriptsize\hair 2pt
 \pinlabel {$0$} [ ] at 63 -18
 \pinlabel {$1$} [ ] at 212 -18
\pinlabel {$2$} [ ] at 365 -18
 \pinlabel {$3$} [ ] at 507 -18
 \pinlabel {$\,[0,1] \times\, D(K)$} [ ] at 181 118
 \pinlabel {$[2,3] \times\; D(0)$} [ ] at 475 118
 \pinlabel {$T\times [-L,L]$} [ ] at -23 5

\endlabellist
\centering
\includegraphics[scale=0.5]{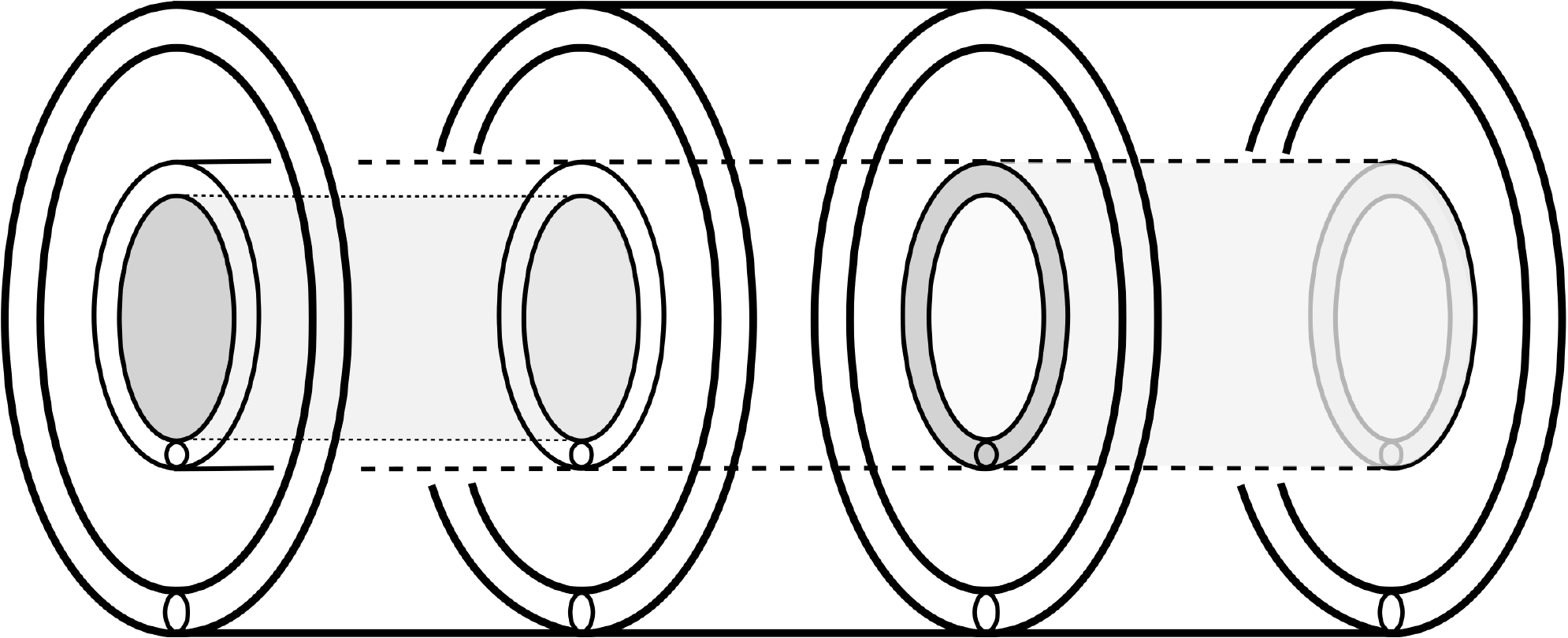}
\caption{$g(W)$}
\label{F:eta}
\end{figure}

The metric $g(W)$ on $W$ is obtained by gluing together the product metrics $[0,3]\times [-L,L]\times g(T)$, $[0,1]\times g_{u,v}$, and $[2,3]\times g_{v,u}$. %Topologically, $W$ is the product manifold $[0,1]\times (D(K)\cup D(0))$. 
The connected component 
$$
(\{1\}\times D(K))\cup ([1,2]\times\{L\}\times  T)\cup (\{2\}\times D(0))
$$
of the boundary of $W$ is the sphere $S^3$ with a non-smooth metric. We modify the metric $g(W)$ in a collar neighborhood of this $S^3$ to obtain a new metric $g'(W)$ that equals the product metric with the standard smooth round metric  $g(S^{3})$. Next, we form the manifold 
$
W'=W\cup ([0,3]\times N)
$
by gluing the two pieces together via the map
$$
\text{id}_{[0,3]\times T}: [0,3]\times \{-L\}\times  T\rightarrow [0,3]\times \partial N.  
$$
The metric $g(W')$ is obtained by gluing together $g'(W)$ and the product metric $[0,3]\times g(N)$. Note that $(W',g(W'))$ is a smooth Riemannian manifold (without corners). It has three boundary components, $(Y,g(Y)_{L})$, $(Y(0),g(Y(0))_{L})$, and $(S^3,g(S^3))$. The signature of $W'$ is zero because the map $H^{2}(W',\partial W'; \mathbb{R})\rightarrow H^{2}(W'; \mathbb{R})$ is trivial. The eta-invariant of $S^3$ vanished because $(S^3,g(S^3))$ admits an orientation-reversing isometry. The Pontryagin form on $[0,3]\times N$ also vanishes with respect to the product metric. Now, using the Atiyah--Patodi--Singer index theorem \cite{APS:I}, we obtain the equality
\[
\eta_{\sign}(Y)-\eta_{\sign}(Y(0))=-\frac{1}{3}\int_{W'}p_{1}(g'(W))  
\]

\smallskip\noindent
whose right hand side is a constant independent of $N$. To compute this constant, we can set $N=D^{2}\times S^{1}$ and $g(N)=g_{u,v}$. Then both $(Y,g(Y)_{L})$ and $(Y(0),g(Y(0))_{L})$ admit orientation reversing isometries hence their eta-invariants vanish. Thus we conclude that, for any $N$, the equality 
$\eta_{\sign}(Y)-\eta_{\sign}(Y(0))=0$ holds.
\end{proof}

\begin{lem}\label{rho}
One has the following identity
\[
\eta_{\sign} (\Sigma(0)) = n\cdot \eta_{\sign} (Y(0)) - \sum_{m=1}^{n-1}\; \sign_{m/n}(K).
\]
\end{lem}

\begin{proof}
%Since $\Omega_3(\Z) = \Omega_3 \oplus \Omega_2 = 0$, we know that $Y(0) = \partial X$ where the isomorphism $H_1(Y(0)) \to \Z$ extends to $H_1(X)$.  Then the $n$-fold covering $\Sigma(0) \to Y(0)$ extends to a cyclic covering $X_n \to X$. Extending the metric on $Y(0)$ to one on $X$ that is a product near the boundary, we have
%\begin{align*} 
%\sign(X) &= \frac13 \int_X p_1 - \etasig(Y(0))\\
%\sign(X_n) &= \frac{n}3 \int_X p_1 - \etasig(\Sigma(0)).
%\end{align*} 
The manifold $\Sigma(0)$ is a regular (unbranched) $n$-fold cyclic cover of $Y(0)$. By splitting the forms on $\Sigma(0)$ into eigenspaces for the induced action of $\tau$, we obtain the relation
$$
\etasig(\Sigma(0)) = \etasig(Y(0)) + \sum_{\alpha \neq 1} \; \etasiga(Y(0)).
$$
Since $\rho_\alpha(Y(0)) = \etasig(Y(0)) - \etasiga(Y(0))$, this immediately implies that
$$
\etasig(\Sigma(0)) = n \cdot \etasig(Y(0)) - \sum_{\alpha \neq 1}\; \rho_\alpha(Y(0)).
$$
%This relation can also be proved by extending the covering $\Sigma(0) \to Y(0)$ over a 4-manifold and applying APS.
Now, if $\alpha$ sends the meridian of $K$ to $m \in \Z/n$, it follows from Gilmer~\cite[Theorem 3.6]{gilmer:config} that $\rho_{\alpha}(Y(0)) = \sign_{m/n}(K)$. Gilmer only states his result for $Y = S^3$ but his method will work more generally for any integral homology sphere.
\end{proof}

\begin{cor}\label{signature-eta compare}
The eta-invariants of $\Sigma$ and $Y$ are related by the formula
\[
\eta_{\sign}(\Sigma) = n\cdot \eta_{\sign}(Y) - \sum_{m=1}^{n-1}\; \sign_{m/n}(K). 
\]
\end{cor}

%%%%%%%%%%%%%%%%%%%%%%%%%%%%%%%%%%%%%%%%%%%%%%%%%%

\section{Proof of Theorem \ref{T:B}}\label{S:lsw}
Let $X$ be the mapping torus of $\tau: \Sigma \to \Sigma$. Since $\tau$ has finite order, $X$ admits an obvious circle action and the orbifold circle bundle $\pi: X \to Y^o$. Let $i\eta$ be the connection form of this bundle. Any orbifold metric $g(Y^o)_L$ on $Y^o$ as in Proposition \ref{gluing2} defines a metric $g = \eta^2 + \pi^*g(Y^o)_L$ on $X$. Similarly, a perturbation form $\omega^o_{\delta} - d\alpha$ as in Proposition \ref{gluing2} lifts to a $\tau$--invariant form $\omega$ on $\Sigma$. Since $\Sigma$ is a rational homology sphere, we can write $\omega = d\gamma$ for some 1-form $\gamma$ on $\Sigma$
and let $\beta = \pi^*\gamma$. With respect to the metric $g$ and perturbation $\beta$, we conclude as in Baldridge \cite{Baldridge} that
\[
\#\M^{\irr} (X,g,\beta) = \sum_{k=0}^{n-1} \; \#\M^{\irr} (Y^o,\s_k,g(Y^o)_L,\omega^o_{\delta} - d\alpha).
\]

\begin{rmk}
A fine point in this kind of argument is the discrepancy between the spinorial connection and the Levi--Civita connection associated with the metric. This discrepancy does not arise in our case because the two Dirac operators in question differ by the Clifford multiplication by the 3-form $\eta \wedge d\eta$ (see Baldridge \cite[Lemma 17]{Baldridge}) which vanishes because $\eta$ is a flat connection.
\end{rmk}

\begin{rmk}
It follows from Baldridge \cite{Baldridge} that the irreducible part of the Seiberg--Witten moduli space on $X$ is non-degenerate. The definition of $\lsw(X)$ in \cite{MRS} further requires that the (perturbed) blown up Seiberg--Witten moduli space on $X$ be free of reducibles. That this is the case for our mapping torus $X$ follows by a simple Fourier analysis argument from a similar property of the blown up Seiberg--Witten moduli space on $Y^o$.
\end{rmk}

\noindent
Proposition \ref{orbifold versus manifold} now implies that
\begin{multline}\label{seiberg-witten count}
\qquad \#\M^{\irr} (X,g,\beta) = \\ n\, \#\M^{\irr}(Y,g_L,\beta_{\delta})\,  -\, \sum_{k=0}^{n-1}\; \SF (\D_{\theta(N^*,\,1-\delta)}, \D_{\theta(N^*,\, 2k/n+1-\delta)}).\qquad
\end{multline}
 The correction terms of Lim \cite{Lim2} and Mrowka--Ruberman--Saveliev \cite{MRS} in the product end case are, respectively, 
\[
c(Y,g_L,\beta_{\delta}) = \frac{1}{2}\,\eta(\D_{\theta(Y)}) + \frac{1}{8}\,\eta_{\text{sign}}(Y)
\]
and
\[
w(X,g,\beta) = \frac{1}{2}\,\eta(\D_{\theta(\Sigma)}) + \frac{1}{8}\,\eta_{\text{sign}}(\Sigma).
\]

\medskip\noindent
It follows from Corollary \ref{dirac eta compare 3} and Corollary \ref{signature-eta compare} that there is a constant $C$, which is independent of $Y,\Sigma,N,K,L$ but may depend on $\delta$ such that 
\begin{multline}\label{E:w}
\qquad w(X,g,\beta) = n\cdot c(Y,g_L,\beta_{\delta}) + \sum_{k=0}^{n-1}\; \SF (\D_{\theta(N^*,\,1-\delta)}, \D_{\theta(N^*,\, 2k/n+1-\delta)}) \\ -\,\frac 1 8\,\sum_{k=0}^{n-1}\; \sign_{k/n} (K) + C + o(1).\quad
\end{multline}

\smallskip\noindent
Combining formulas \eqref{seiberg-witten count} and \eqref{E:w} with the Lim formula \cite{Lim2} for the Casson invariant, $-\lambda (Y) = \#\M^{\irr} (Y,g_L,\beta_{\delta}) + c(Y,g_L,\beta_{\delta})$, we obtain\footnote{Our orientation conventions differ from those of Lim \cite{Lim1}, hence the Casson invariant $\lambda(Y)$ shows up with the negative sign. This is the same issue we dealt in \cite[Section 7]{RS:jktr}.}
\begin{multline}\notag
\lsw(X) = \# \M^{\irr}(X,g,\beta) + w (X,g,\beta) \\ = -n\cdot \lambda (Y) - \frac 1 8\cdot \sum_m\; \sign_{m/n} (K) + C + o(1).
\end{multline}
Since $\lsw(X)$, $\lambda(Y)$, $\sign_{k/n} (K)$ are metric independent and $C$ is independent of $L$, by passing to the limit as $L\rightarrow \infty$, we obtain the formula
\[
\lsw (X) = -n\cdot \lambda(Y) - \frac 1 8\cdot \sum_m\;\sign_{m/n} (K) + C,
\]
which further implies that $C$ is a truly universal constant (independent, in particular, of $Y$, $K$, and $\delta$). By applying this formula to $Y = S^3$ and an unknot $K$, we conclude that $C = 0$. This finishes the proof.

%%%%%%%%%%%%%%%%%%%%%%%%%%%%%%%%%%%%%%%%%%%%%%%%%%

\section{Computing the Furuta--Ohta invariant}\label{S:lfo}
Let as before $K$ be a knot in an integral homology sphere $Y$, and $\Sigma$ the $n$--fold cyclic cover of $Y$ with branch set $K$ for $n \ge 2$. We continue to assume that $\Sigma$ is a rational homology sphere. It comes equipped with the covering translation $\tau: \Sigma \to \Sigma$ which is an orientation preserving diffeomorphism of order $n$. The mapping torus of $\tau$ is the smooth 4-manifold $X = ([0,1]\times\Sigma)\,/\,(0,x)\sim (1,\tau(x))$ with the product orientation. 

\begin{pro}
The manifold $X$ has the integral homology of $S^1 \times S^3$.
\end{pro}

\begin{proof}
Let $N$ be the knot $K$ exterior, $N_n \to N$ the $n$--fold cyclic cover of $N$, and $N_{\infty} \to N$ its infinite cyclic cover. %Note that $H_1 (N_n) = N_1 (\Sigma)\,\oplus\,\Z$, where $\Z$ is generated by a lift of the meridian. 
Denote by $\tau: N_{\infty} \to N_{\infty}$ the covering translation; it descends to the covering translations $\tau: N_n \to N_n$ and $\tau: \Sigma \to \Sigma$. It follows from the Wang exact sequence 
\[
\xymatrix{
\ldots \ar[r] & H_1(\Sigma) \ar[r]^(.5){\tau_*-1} &H_1(\Sigma) \ar[r] & H_1(X) \ar[r] & H_0(\Sigma) = \Z \ar[r]^(.5){\tau_*-1} &H_0(\Sigma) = \Z
}
\]
that all we need to prove is to show that $\tau_*-1:  H_1(\Sigma) \to  H_1(\Sigma) $ is surjective. To this end, choose a Seifert surface for $K$ and a basis for its first homology yielding the $2h \times 2h$ Seifert matrix $V$.  Then we have the following presentations~\cite[Theorem 8.8 and Proposition 8.20]{burde-zieschang:knots} for the first homology of $N_{\infty}$ and $\Sigma$ as modules over, respectively, $\Z[t,t^{-1}]$ and $\Z[\Z/n]$:
\begin{align*}
H_1(N_{\infty}) = & \left(\Z[t,t^{-1}]\right)^{2h}/(V^{\top} - t V), \\
H_1(\Sigma) = & \left(\Z[\Z/n]\right)^{2h}/(V^{\top} - t V).
\end{align*}
It is a standard fact \cite{levine:modules} that $\tau_* -1: H_1(N_{\infty}) \to H_1(N_{\infty})$ is an isomorphism. Writing $\pi$ for the projection homomorphism 
\[
\pi: \Z[t,t^{-1}] \longrightarrow  \Z[t,t^{-1}]/(t^n-1) = \Z[\Z/n],
\]
we obtain a commutative diagram
\[
\xymatrix{
\left(\Z[t,t^{-1}]\right)^{2h} \ar[rr]^{V^{\top} - t V}\ar[d]^{\pi} &&  \left(\Z[t,t^{-1}]\right)^{2h}\ar[d]^{\pi} \\
\left(\Z[\Z/n]\right)^{2h} \ar[rr]^{V^{\top} - t V} && \left(\Z[\Z/n]\right)^{2h}
}
\]
which readily implies that the map $\tau_* -1:  H_1(\Sigma) \to  H_1(\Sigma) $ is surjective.
%Let $k$ be a knot in $S^3$. Its $n$--twist spun knot is a 2--knot in $S^4$ whose exterior was described by Zeeman \cite[Main Theorem]{Z} as the fiber bundle over a circle whose fiber is the punctured $\Sigma$ and whose monodromy is $\tau$. One can easily see from that description that $X$ is obtained from $S^4$ by surgery along the $n$--twist spun knot and is therefore an integral homology $S^1 \times S^3$. A modification of Zeeman's argument proves the claimed result for knots in an arbitrary integral homology sphere $Y$ (this may need an explanation)
\end{proof}

\begin{cor}\label{C:fix}
The automorphism $\tau_*: H_1 (\Sigma) \to H_1 (\Sigma)$ has zero fixed point set.
\end{cor}

\begin{rmk}\label{R:unique}
The manifold $X$ has a unique $\spinc$ structure. Its restriction to $\Sigma$ is the unique $\spinc$ structure $\s$ with the property $\tau_*(\s) = \s$ which appears in the statement of Theorem \ref{T:A}. This $\spinc$ structure is in fact a spin structure, obtained by restricting to $\Sigma$ either of the two distinct  spin structures on $X$; see the discussion at the beginning of Section \ref{S:dirac}.
\end{rmk}

The infinite cyclic cover $\tilde X \to X$, which is just the product $\mathbb R \times \Sigma$, has the rational homology of $S^3$, hence the conditions \eqref{E:zz} are satisfied. According to \cite[Section 7.1]{LRS2}, the manifold $X$ has a well-defined Furuta--Ohta invariant,
\begin{equation}\label{E:lfo}
\lfo (X)\; = \; \frac 1 4\, \#\M^* (X)\, \in\, \Q,
\end{equation}
where $\#\M^* (X)$ stands for the signed count of points in the (possibly perturbed) moduli space $\M^* (X)$ of irreducible ASD connections in a trivial $SU(2)$--bundle $E \to X$, with the signs determined by a choice of orientation and homology orientation on $X$. In this section, we will prove the following formula for this invariant.

\begin{thm}\label{T:main}
Let $\lambda (Y)$ be the Casson invariant of $Y$, and denote by $\sign_{m/n} (K)$ the Tristram--Levine equivariant signatures of the knot $K$. Then
\[
\lfo (X)\; =\; n\cdot\lambda(Y)\; + \; \frac 1 8\,\cdot  \sum_{m = 1}^{n-1}\; \sign_{m/n} (K). 
\]
\end{thm}

\noindent
This formula was proved in \cite{CS} and \cite{RS2} under the assumption $\Sigma$ is an integral homology sphere, and in \cite{LRS2} under the assumption that $n = 2$. Our proof here will rely on the extension of those techniques to the general case at hand.

%%%%%%%%%%%%%%%%%%%%%%%%%%%%%%%%%%%%%%%%%%%%%%

\subsection{Equivariant theory}\label{S:eq}
We will first describe $\M^*(X)$ in terms of $\R(\Sigma)$, the $SU(2)$--character variety of $\pi_1(\Sigma)$. To this end, consider the splitting
\[
\R(\Sigma)\; =\; \{\theta\}\, \sqcup\, \R_{\ab} (\Sigma)\, \sqcup\, \R_{\irr}(\Sigma),
\]
whose three components consist of the trivial representation $\theta$ and the conjugacy classes of abelian (that is, non-trivial reducible) and irreducible representations, respectively. This decomposition is preserved by the map $\tau^*: \R(\Sigma) \to \R (\Sigma)$ induced by the covering translation. Denote by $\R^{\tau} (\Sigma)$ the fixed point set of the map $\tau^*$ acting on $\R(\Sigma)\setminus \{\theta\} = \R_{\ab}(\Sigma)\, \sqcup\, \R_{\irr} (\Sigma)$ so that
\[
\R^{\tau} (\Sigma)\; =\; \R_{\ab}^{\tau} (\Sigma)\, \sqcup\, \R_{\irr}^{\tau} (\Sigma).
\]
The following algebraic lemma will allow us to obtain useful information about the action of $\tau^*$ on $\R_{\ab} (\Sigma)$.

\begin{lem}\label{L:ab}
Let $G$ be a finite abelian group and $s: G \to G$ an automorphism with zero fixed point set. Then the induced automorphism $s^*: \Hom (G,U(1)) \to \Hom (G,U(1))$ of the character group of $G$ has the trivial character as its only fixed point.
\end{lem}

\begin{proof}
Let us consider the homomorphism $u: G \to G$ given by the formula $u (g) = s (g) - g$. Since $s$ has zero fixed point set, $u$ is injective and, since $G$ is finite, $u$ is an isomorphism. This implies that $u^*$ is also an isomorphism, which completes the proof because $u^*(\chi) = s^*(\chi) \cdot \chi^{-1}$ on characters $\chi: G \to U(1)$.
\end{proof}

The significance of the character group $\Hom (H_1(\Sigma),U(1))$ to us is that its quotient by the equivalence relation identifying $\alpha$ with $\alpha^{-1}$ is precisely $\R_{\ab} (\Sigma)$. We call a character $\alpha \in \Hom (H_1 (\Sigma), U(1))$ \emph{central} if the representation obtained by composing $\alpha$ with the inclusion of $U(1)$ in $SU(2)$ as a maximal torus is central; the latter simply means that the image of $\alpha$ is contained in $\{\pm 1\}\subset U(1)$.

\begin{pro}\label{P:ab}
The fixed point set $\R_{\ab}^{\tau} (\Sigma)$ consists of the equivalence classes of non-central characters $\alpha \in \Hom (H_1 (\Sigma), U(1))$ such that $\tau^*\alpha = \alpha^{-1}$. In particular, $\R_{\ab}^{\tau} (\Sigma)$ is empty for odd $n$, and $\R_{\ab}^{\tau} (\Sigma) = \R_{\ab} (\Sigma)$ for $n = 2$.
\end{pro}

\begin{proof} 
A character $\alpha$ gives rise to a point in $\R_{\ab}^{\tau} (\Sigma)$ if and only if $\tau^*\alpha = \alpha$ or $\tau^*\alpha = \alpha^{-1}$. The former only occurs for the trivial character by Lemma \ref{L:ab}. For a central character $\alpha$, the condition $\tau^*\alpha = \alpha^{-1}$ is equivalent to $\tau^*\alpha = \alpha$ hence $\alpha$ is again trivial. If $n$ is odd, it follows that $\alpha = \alpha^{-1}$ so $\alpha$ is central and hence trivial. If $n = 2$, it follows from Lemma \ref{L:ab} that $\tau^*$ acts as the negative identity on the character group and therefore as the identity on $\R_{\ab} (\Sigma)$.
\end{proof}

%The space $\R^{\tau}(\Sigma)$ is called non-degenerate if its Zariski tangent space vanishes at every point. If $\R^{\tau}(\Sigma)$ is non-degenerate then it consists of finitely many points which can be counted with signs as in \cite{CS} to obtain the equivariant Casson invariant $\lambda^{\tau} (\Sigma)$ (if $\tau = \id$, it coincides with the regular Casson invariant $\lambda(\Sigma)$). Under the additional assumption that $\Sigma/\tau$ is an integral homology sphere, the modulo 2 reduction of $\lambda^{\tau}(\Sigma)$ is the Rokhlin invariant.

%\begin{thm}\label{T1}
%If the space $\R^{\tau}(\Sigma)$ is non-degenerate then $\M_0^*(X)$ consists of finitely many points and 
%\[
%\lambda^{\tau}(\Sigma) = \frac 1 4\,D_0(X),
%\]
%where $D_0(X)$ is the degree zero Donaldson polynomial of $X$ obtained by counting points in $\M_0^*(X)$ with signs. 
%\end{thm}

\begin{pro}\label{P:one}
Let $i: \Sigma \to X$ be the inclusion map given by the formula $i(x) = [0,x]$. Then the induced map
\begin{equation}\label{E:i*}
i^*: \M^*(X) \to \R^{\tau}(\Sigma)
\end{equation}
is well-defined, and is a one-to-one correspondence over $\R_{\ab}^{\tau} (\Sigma)$ and a two-to-one correspondence over $\R^{\tau}_{\irr} (\Sigma)$.
\end{pro}

\begin{proof}
The natural projection $X\to S^1$ is a locally trivial bundle whose 
homotopy exact sequence
\[
\begin{CD}
0 @>>> \pi_1(\Sigma) @>>> \pi_1(X) @>>> \mathbb Z @>>> 0
\end{CD}
\]
splits, making $\pi_1(X)$ into a semi-direct product of $\pi_1(\Sigma)$ and $\mathbb Z$. Let $t$ be a generator of $\mathbb Z$ then every representation $A: \pi_1(X)\to SU(2)$ determines and is uniquely determined by the pair $(\alpha,u)$ where $u = A(t)$ and $\alpha= i^* A: \pi_1(\Sigma)
\to SU(2)$ is a representation such that $\tau^*\alpha = u\alpha u^{-1}$. In particular, the conjugacy class of $\alpha$ is fixed by $\tau^*$. 

If $\alpha$ is trivial, $A$ must be reducible. If $\alpha$ is non-trivial abelian, it cannot be central by Proposition \ref{P:ab}. Given a non-central abelian $\alpha$, conjugate it to a representation whose image is in the group $U(1)$ of unit complex numbers in $SU(2)$. Then $\tau^*\alpha = \alpha^{-1}$ by Proposition \ref{P:ab}, hence $\alpha = i^*A$ with $u = A(t)$ in the circle $j\cdot U(1)$. In particular, $A$ is irreducible and $u^2 = -1$. Since any two elements of $j\cdot U(1)$ are conjugate to each other by a unit complex number, the map $i^*$ is a one-to-one correspondence over $\R_{\ab}^{\tau} (\Sigma)$. Finally, let $\alpha$ be an irreducible representation with the conjugacy class in $\R^{\tau}_{\irr} (\Sigma)$. Then there is a unit quaternion $u$ such that $\tau^*\alpha = u\alpha u^{-1}$, and therefore $\alpha$ is in the image of $i^*$. Moreover, there are exactly two different choices of $u$ such that $\tau^*\alpha= u\alpha u^{-1}$ because if $u_1\alpha u_1^{-1} = u_2\alpha u_2^{-1}$ then $u_1 = \pm\, u_2$ since $\alpha$ is irreducible. Therefore, the map $i^*$ is a two-to-one correspondence in this case. Also note that the irreducibility of $\alpha$ implies that $u^n = \pm 1$. 
\end{proof}

\begin{rmk}
It follows from the above proof that the characters in $\M^*(X)$ that are mapped by $i^*$ to $\R^{\tau}_{\ab} (\Sigma)$ are binary dihedral, while those mapped to $\R^{\tau}_{\irr}(\Sigma)$ are not. 
\end{rmk}

The Zariski tangent space to $\R^{\tau}(\Sigma)$ at a point $[\alpha] \in \R^{\tau}(\Sigma)$ is the fixed point set of the map $\tau^*: T_{[\alpha]} \R (\Sigma) \to T_{[\alpha]} \R (\Sigma)$. Using an identification $T_{[\alpha]} \R (\Sigma) = H^1 (\Sigma,\ad\alpha)$ and the fact that $\tau^*\alpha = u\alpha u^{-1}$, this set can be described in cohomological terms as the fixed point set of the map
\[
\Ad u\circ \tau^*: H^1 (\Sigma,\ad\alpha) \to H^1 (\Sigma,\ad\alpha).
\]
We call $\R^{\tau}(\Sigma)$ {\it non-degenerate} if the equivariant cohomology groups $H^1_{\tau}(\Sigma,\ad\alpha) = \Fix\,\allowbreak (\Ad u\,\circ \tau^*: H^1 (\Sigma,\ad\alpha) \to H^1 (\Sigma,\ad\alpha))$ vanish for all $[\alpha]\in\R^{\tau}(\Sigma)$. The moduli space $\M^*(X)$ is called {\it non-degenerate} if $\coker(d_A^*\,\oplus\, d_A^+) = 0$ for all $[A]\in \M^*(X)$. Since $\ind(d^*\,\oplus\, d_A^+) = \dim\M^*(X) = 0$, this is equivalent to $\ker(d_A^*\,\oplus\, d_A^+) = 0$ and, since $A$ is flat and irreducible, to simply $H^1(X,\ad A) = 0$.

\begin{pro}\label{P2}
The moduli space $\M^*(X)$ is non-degenerate if and only if $\R^{\tau}(\Sigma)$ is non-degenerate. 
\end{pro}

\begin{proof}
The group $H^1 (X,\ad A)$ can be computed with the help of the Leray--Serre spectral sequence of the fibration $X\to S^1$ with fiber $\Sigma$. The $E_2$--page of this spectral sequence is 
\[
E_2^{pq} = H^p(S^1,\H^q(\Sigma,\ad\alpha)),
\]
where $\alpha = i^* A$ and $\H^q(\Sigma,\ad\alpha)$ is the local coefficient system associated with the fibration. The groups $E_2^{pq}$ vanish for all $p\ge 2$ hence the spectral sequence collapses at $E_2$--page, and
\begin{equation}\label{E:E2}
H^1(X,\ad A) = H^1(S^1,\H^0(\Sigma,\ad\alpha))\,\oplus\,H^0(S^1,\H^1(\Sigma,
\ad\alpha)).
\end{equation}
The generator of $\pi_1(S^1)$ acts on the cohomology groups $H^* (\Sigma,\ad\alpha)$ as
\[
\Ad u\circ \tau^*: H^* (\Sigma,\ad\alpha) \to H^* (\Sigma,\ad\alpha),
\]
where $u$ is such that $\tau^*\alpha = u \alpha u^{-1}$. If $\alpha$ is irreducible, $H^0 (\Sigma,\ad\alpha) = 0$ and the first summand in \eqref{E:E2} vanishes. If $\alpha$ is non-trivial abelian, we may assume without loss of generality that it takes values in the group $U(1)$ of unit complex numbers. Then $\tau^*\alpha = u\alpha u^{-1}$ for some $u \in j\cdot U(1)$ and $H^0 (\Sigma,\ad\alpha) = i\cdot \mathbb R$ as a subspace of $\su(2)$, with $\tau^* = \id$. One can easily check that $\Ad u$ acts as minus identity on $i\cdot\mathbb R$ hence the first summand in \eqref{E:E2} again vanishes. The second summand in \eqref{E:E2} is the fixed point set of $\tau^*$ acting on $H^1(\Sigma,\ad\alpha)$, which is the equivariant cohomology $H^1_{\tau}(\Sigma,\ad\alpha)$. Thus we conclude that $H^1(X,\ad A) = H^1_{\tau}(\Sigma,\ad\alpha)$, which completes the proof. 
\end{proof}

Let us assume that $\R^{\tau}(\Sigma)$ is non-degenerate. For any $[\alpha] \in \R^{\tau}(\Sigma)$, its orientation will be given by
\[
(-1)^{\sf^{\tau}(\theta,\alpha)}
\]
where $\sf^{\tau} (\theta,\alpha)$ is the mod 2 equivariant spectral flow defined in \cite[Section 3.4]{RS2} for irreducible $\alpha$. That definition extends word for word to abelian $\alpha$ after one resolves the technical issue of the existence of a constant lift, which we will do next.

Let $P$ be an $SU(2)$ bundle over $\Sigma$ with a fixed trivialization and $\alpha$ an abelian flat connection in $P$ with holonomy in $\R^{\tau}_{\ab}(\Sigma)$; we are abusing notations by using the same symbol for the connection and its holonomy. Since $\R^{\tau}_{\ab}(\Sigma)$ is empty for odd $n$ (see Proposition \ref{P:ab}) we will assume without loss of generality that $n$ is even. Then $\tau$ admits a lift $\tilde \tau: P \to P$ such that $\tilde\tau^*\alpha = \alpha$. Since $\alpha$ is abelian and non-central, this lift is defined uniquely up to the stabilizer of $\alpha$, which is a copy of $U(1)$ in $SU(2)$. The lift $\tilde\tau$ can be written in the base-fiber coordinates as $\tilde \tau(x,y) = (\tau(x), \rho(x)\cdot y)$ for some function $\rho: \Sigma \to SU(2)$. We call it \emph{constant} if there exists $u\in SU(2)$ such that $\rho(x) = u$ for all $x\in SU(2)$.

\begin{lem}\label{L:lift}
By changing $\alpha$ within its gauge equivalence class, one may assume that $\tilde\tau$ is a constant lift with $u^2 = -1$.
\end{lem}

\begin{proof}
The equation $\tilde\tau^*\alpha = \alpha$ implies that $(\tilde\tau^n)^*\alpha = \alpha$, so the gauge transformation $\tilde\tau^n$ belongs to the stabilizer of the connection $\alpha$. If $x \in \Fix (\tau)$ then $\tilde\tau^n (x,y) = (x,\rho(x)^n\cdot y)$ hence $\rho(x)^n$ is a unit complex number independent of $x$. This implies that $\rho(x)$ itself is a unit complex number unless $\rho(x)^n = \pm 1$. The latter equation actually implies that $\rho(x)^2 = -1$ because, at the level of holonomy representations, $\tau^*\alpha = \alpha^{-1}$ is conjugate to $\alpha$ by an element $u \in SU(2)$ with $u^2 = -1$; see the proof of Proposition \ref{P:one}. Since $\rho(x)^2 = -1$ describes a single conjugacy class $\tr \rho(x) = 0$ in $SU(2)$, we may assume that $\rho(x) = u$ for all $x \in \Fix(\tau)$. 

To finish the proof, we will follow the argument of \cite[Section 2.2]{RS2}. Let $u: P \to P$ be the constant lift $u(x,y) = (\tau(x),u\cdot y)$ and consider the $SO(3)$ orbifold bundles $P/\tilde\tau$ and $P/u$ over the integral homology sphere $Y$. All such bundles are classified by the holonomy around the singular set in $Y$. Since this holonomy equals $\Ad(u)$ in both cases, the bundles $P/\tilde\tau$ and $P/u$ must be isomorphic, with any isomorphism pulling back to a gauge transformation $g: P \to P$ relating the lifts $\tilde\tau$ and $u$.
\end{proof}

\begin{pro}
Assuming that the moduli space $\R^{\tau}(\Sigma)$ is non-degenerate, the map \eqref{E:i*} is orientation preserving. 
\end{pro}

\begin{proof}
The proof from \cite[Section 3]{RS2} extends to the current situation with no change.
\end{proof}

%%%%%%%%%%%%%%%%%%%%%%%%%%%%%%%%%%%%%%%%%%%%%%%

\subsection{Orbifold theory}\label{S:orb}
Under the continued non-degeneracy assumption, we will now describe $\R^{\tau}(\Sigma)$ in terms of orbifold representations. Let us consider the orbifold fundamental group $\pi_1^V (Y,K) = \pi_1 (N)/ \langle \mu^n \rangle$, where $N = Y - \Int(D(K))$ is the knot exterior and $\mu$ is a meridian of $K$. This group can be included into the split orbifold exact sequence 
\[
\begin{tikzpicture}
\draw (1,1) node (a) {$1$};
\draw (3,1) node (b) {$\pi_1 \Sigma$};
\draw (6,1) node (c) {$\pi_1^V(Y, K)$};
\draw (9,1) node (d) {$\Z/n$};
\draw (11,1) node (e) {$1,$};
\draw[->](a)--(b);
\draw[->](b)--(c) node [midway,above](TextNode){$\pi_{*}$};
\draw[->](c)--(d) node [midway,above](TextNode){$j$};
\draw[->](d)--(e);
\end{tikzpicture}
\]
where $j$ is the abelianization homomorphism. Denote by $\R^V (Y,K; SO(3))$ the character variety of irreducible $SO(3)$ representations of the group $\pi_1^V (Y, K)$, and also introduce the character variety $\R^{\tau} (\Sigma; SO(3))$ of non-trivial representations $\pi_1 \Sigma \to SO(3)$ fixed by $\tau^*$.

\begin{pro}\label{P:eq-knots}
The pull back of representations via the map $\pi_*$ in the orbifold exact sequence gives rise to a one-to-one correspondence
\[
\pi^*: \R^V (Y,K;SO(3)) \longrightarrow \R^{\tau} (\Sigma;SO(3)).
\]
\end{pro}

\begin{proof}
One can easily see that a representation $\alpha': \pi_1^V(Y,K) \to SO(3)$ pulls back to a trivial representation $\theta: \pi_1 \Sigma \to SO(3)$ if and only if $\alpha'$ is reducible. The same argument as in \cite[Proposition 3.3]{CS} shows that all pull-back representations belong to $\R^{\tau} (\Sigma,SO(3))$. The inverse map for $\pi^*$ is constructed as follows: given $[\alpha] \in \R^{\tau} (\Sigma,SO(3))$ choose $v \in SO(3)$ such that $\tau^*\alpha = v \alpha v^{-1}$, and define a representation $\alpha'$ of $\pi_1^V(Y,K) = \pi_1 \Sigma \rtimes \Z/n$ by the formula 
\begin{equation}\label{E:inverse}
\alpha'(g \cdot \mu^k)\, =\, \alpha(g)\cdot v^k.
\end{equation}
If $\alpha$ is irreducible, the element $v$ is unique hence formula \eqref{E:inverse} gives an inverse map. If $\alpha$ is non-trivial abelian, an argument similar to that in the proof of Proposition \ref{P:ab} shows that $v = \Ad u$ for some $u \in j\cdot U(1)$. Since any two elements of $j\cdot U(1)$ are conjugate to each other by a unit complex number, formula \eqref{E:inverse} again gives an inverse map.
\end{proof}

Representations $\pi_1^V (Y,K) \to SO(3)$ need not lift to $SU(2)$ representations; however, they lift to projective representations $\pi_1^V (Y,K) \to SU(2)$ sending $\mu^n$ to $\pm 1$; see \cite[Section 3.1]{RS1}. The character variety of such projective representations will be denoted by $\R^V (Y,K)$, and it will be oriented using the orbifold spectral flow.

\begin{pro}
The correspondence of Proposition \ref{P:eq-knots} gives rise to an orientation preserving correspondence $\R^V (Y,K) \to \R^{\tau} (\Sigma)$ which is one-to-one over $\R_{\ab}^{\tau} (\Sigma)$ and two-to-one over $\R^{\tau}_{\irr} (\Sigma)$.
\end{pro}

\begin{proof}
We first need to check that the map $\R^V (Y,K) \to \R^{\tau} (\Sigma)$ is well-defined because the pull back of a projective representation $\pi_1^V (Y,K) \to SU(2)$ is {\it a priori} a projective representation $\alpha: \Sigma \to SU(2)$. The only obstruction to it being an actual representation is the second Stiefel--Witney class $w_2 (\Ad\alpha) \in H^2 (\Sigma;\Z/2)$. That this class vanishes can be seen as follows. The conjugacy class of $\alpha$ is fixed by $\tau$, therefore, $w_2 (\Ad\alpha)$ must belong to the fixed point set of $\tau^*: H^2 (\Sigma; \Z/2) \to H^2 (\Sigma; \Z/2)$. Using Poincar\'e duality, this fixed point set can be identified with the kernel of the map $\tau_* - 1: H_1 (\Sigma; \Z/2) \to H_1 (\Sigma; \Z/2)$. This kernel vanishes because the map $\tau_* -1: H_1 (\Sigma) \to H_1 (\Sigma)$ is injective by Corollary \ref{C:fix} and is therefore an automorphism of the finite abelian group $H_1 (\Sigma)$; cf. the proof of Lemma \ref{L:ab}.

Let us now consider the adjoint representation $\Ad: SU(2) \to SO(3)$ and the induced maps on character varieties,
\begin{equation}\label{E:maps}
 \R^{\tau}(\Sigma) \to \R^{\tau}(\Sigma; SO(3))\quad\text{and}\quad\R^V (Y,K) \to \R^V (Y,K; SO(3)). 
\end{equation}

The first map is a one-to-one correspondence, which can be seen as follows. We showed in the previous paragraph that every equivariant representation $\pi_1 \Sigma \to SO(3)$ admits a lift to a representation $\pi_1 \Sigma \to SU(2)$. The number of all such lifts is known to equal the cardinality of $H^1 (\Sigma; \Z/2)$. However, we are only interested in equivariant lifts, and their number equals the cardinality of the fixed point set of $\tau^*: H^1 (\Sigma; \Z/2) \to H^1 (\Sigma; \Z/2)$. The latter cardinality is easily seen to be one using Corollary \ref{C:fix}.

The second map in \eqref{E:maps} is the quotient map by the action of $\Z/2$ sending the image of the meridian $\mu$ to its negative. The fixed points of this action are precisely the binary dihedral projective representations $\alpha': \pi_1^V (Y,K) \to SU(2)$. Now, the proof will be finished as soon as we show that an irreducible projective representation $\alpha': \pi_1^V (Y,K) \to SU(2)$ is binary dihedral if and only if its pull back representation $\pi^*\alpha': \pi_1 \Sigma \to SU(2)$ is abelian.

If $\pi^*\alpha'$ is abelian, its image belongs to $U(1) \subset SU(2)$ and the image of $\alpha'$ to its $\Z/n$ extension. Since the only finite extension of $U(1)$ inside $SU(2)$ is the binary dihedral group $U(1)\,\cup\,j\cdot U(1)$, we conclude that $\alpha'$ must be binary dihedral. In particular, if $n$ is odd, the representation $\pi^*\alpha'$ cannot be abelian, which matches the fact that $\R^{\tau}_{\ab} (\Sigma)$ is empty by Proposition \ref{P:ab}. Conversely, it follows from the orbifold exact sequence that $\pi_1 \Sigma$ is the commutator subgroup of $\pi_1^V (Y,K)$. Therefore, if $\alpha'$ is binary dihedral, the image of $\pi^*\alpha'$ must belong to the commutator subgroup of $U(1)\,\cup\,j\cdot U(1)$, which is of course the group $U(1)$.

Since the orbifold spectral flow matches the equivariant spectral flow used to orient $\R^{\tau}(\Sigma)$, the above correspondence is orientation preserving.
\end{proof}

%%%%%%%%%%%%%%%%%%%%%%%%%%%%%%%%%%%%%%%%%%%%%%%

\subsection{Perturbations}\label{S:pert}
In this section, we will remove the assumption that $\R^{\tau}(\Sigma)$ is non-degene\-rate which we used until now. To accomplish that, we will switch from the language of representations to the language of connections. Let $P$ a trivialized $SU(2)$ bundle over $\Sigma$. Any endomorphism $\tilde\tau: P \to P$ which lifts the involution $\tau$ induces an action on the space of connections $\A(\Sigma)$ by pull back. Since any two such lifts are related by a gauge transformation, this action gives a well defined action on the configuration space $\B(\Sigma) = \A(\Sigma)/\G(\Sigma)$. The fixed point set of this action will be denoted by $\B^{\tau}(\Sigma)$. 

The irreducible part of $\B^{\tau}(\Sigma)$ was studied in \cite{RS2} hence we will only deal with reducible connections. In fact, we will further restrict ourselves to constant lifts $u$ because any flat abelian connection $\alpha$ admits such a lift; see Lemma \ref{L:lift}. 

Let $\A^u (\Sigma) \subset \A(\Sigma)$ consist of all non-trivial connections $A$ such that $u^*A = A$, and $\G^u(\Sigma) \subset \G(\Sigma)$ of all gauge transformations $g$ such that $gu = ug$. The quotient space $\A^u(\Sigma)/\G^u(\Sigma)$ will be denoted by $\B^u(\Sigma)$. The following lemma is a key to making the arguments of \cite{RS2} work in the case of abelian connections. 

\begin{lem}
The group $\G^u(\Sigma)$ acts on $\A^u(\Sigma)$ with the stabilizer $\{\pm 1\}$. Moreover, the natural map $\B^u(\Sigma) \to \B^{\tau}(\Sigma)$ is a two-to-one correspondence to its image on the irreducible part of $\B^u(\Sigma)$, and a one-to-one correspondence on the reducible part.
\end{lem}

\begin{proof}
For the sake of simplicity, we will assume that reducible connections have their holonomy in the subgroup $U(1)$ of unit complex numbers in $SU(2)$, and that $u \in j\cdot U(1)$. Let us suppose that $g^*A = A$ for a connection $A \in \A^u (\Sigma)$ and a gauge transformation $g \in \G^u (\Sigma)$. If $A$ is irreducible, we automatically have $g = \pm 1$. If $A$ is non-trivial abelian, then $g$ is a complex number, and the condition $ug = gu$ implies that $g = \pm 1$.

To prove the second statement, consider a connection $A$ such that $u^*A = A$ and consider its gauge equivalence class in $\B^{\tau}(\Sigma)$. It consists of all connections $g^*A$ such that $u^* g^* A = g^* A$. Since $A = u^* A$, we immediately conclude that $u^* g^* A = g^* u^* A$ so that $ug$ and $gu$ differ by an element in the stabilizer of $A$. If $A$ is irreducible, its stabilizer consists of $\pm 1$ hence $ug = \pm g u$. The group of gauge transformations satisfying this condition contains $\G^u(\Sigma)$ as a subgroup of index two, which leads to the desired two-to-one correspondence. If $A$ is non-trivial abelian, its stabilizer consists of unit complex numbers. Therefore, we can write $ug = c^2 gu$ with $c \in U(1)$ or, equivalently, $ucg = cgu$. This provides us with a gauge transformation $cg \in \G^u (\Sigma)$ such that $(cg)^* A = g^*A$, yielding the one-to-one correspondence on the reducible part.
\end{proof}

With this lemma in place, the proof of Proposition \ref{P:one} can be re-stated in gauge-theoretic terms as in \cite[Proposition 3.1]{RS2}. The treatment of perturbations in our case is then essentially identical to that in \cite{CS} and \cite{RS2}, one important observation being that the orbifold representations $\alpha'$ that pull back to abelian representations of $\pi_1 (\Sigma)$ are in fact irreducible. This fact is used in the proof of \cite[Lemma 3.8]{CS}, which supplies us with sufficiently many admissible perturbations. 

%%%%%%%%%%%%%%%%%%%%%%%%%%%%%%%%%%%%%%%%%%%%%%%

\subsection{Proof of Theorem \ref{T:main}}
The outcome of Section \ref{S:eq} and Section \ref{S:orb} is that, perhaps after perturbing as in Section \ref{S:pert}, we have two orientation preserving correspondences,
\[
\M^*(X)\; \longrightarrow\; \R^{\tau}(\Sigma)\; \longleftarrow\; \R^V (Y,K),
\]
both of which are one-to-one over $\R_{\ab} (\Sigma)$ and two-to-one over $\R^{\tau}_{\irr} (\Sigma)$ (we omit perturbations in our notations). These correspondences imply the existence of an orientation preserving one-to-one correspondence between $\M^*(X)$ and $\R^V (Y,K)$. The proof of Theorem \ref{T:main} will be complete after we express the signed count of points in $\R^V (Y,K)$ in terms of the Casson invariant of $Y$ and the equivariant knot signatures of $K$.

The character variety $\R^V (Y,K)$ of projective representations $\alpha'$ splits into two components corresponding to whether $(\alpha'(\mu))^n$ equals $+1$ or $-1$. Let $N$ be the exterior of the knot $K$ then this splitting corresponds to the splitting
\begin{equation}\label{E:ss}
\R^V (Y,K)\; = \; \bigcup_{k=0}^n\; \S_{k/n} (N,SU(2)),
\end{equation}
where $\S_a (N,SU(2))$ comprises the conjugacy classes of representations $\gamma: \pi_1 N \to SU(2)$ such that $\tr \gamma (\mu) = 2\cos (\pi a)$. According to Herald \cite{H}, see also Collin--Saveliev \cite{CS}, the combined signed count of points in \eqref{E:ss} equals 
\[
4n\cdot\lambda (Y)\; +\; \frac 1 2\cdot \sum_{m=1}^{n-1} \; \sign_{m/n} (K).
\]
Dividing this formula by four, we obtain the formula for the Furuta--Ohta invariant $\lfo (X)$ claimed in Theorem \ref{T:main}.

%%%%%%%%%%%%%%%%%%%%%%%%%%%%%%%%%%%%%%%%%%%%%%%%%%

\section{Applications}\label{S:applications}
In this section we supply proofs for the applications of our main theorem discussed in Section~\ref{introapps} of the introduction; for the convenience of the reader we will restate each result before giving the proof.

%We will show that the property of not having an L-space branched cover can sometimes hold for an entire concordance class of knots.

\begin{thm}[Theorem~\ref{T:Donc} from the introduction]\label{conc}
Let $n=p^m$ for $p$ a prime number. There is a knot $K_n$ such that, for any knot $K$ that is smoothly concordant to $K_n$, its $n$-fold cyclic branched cover $\Sigma_n (K)$ is not an $L$-space.
\end{thm}

\begin{proof}
Since $n$ is a prime power, it is standard that all of the Tristram--Levine signatures are knot concordance invariants.  Similarly, the $h$-invariant of the $n$-fold cyclic branched cover, with the specified $\spinc$ structure, is a knot concordance invariant~\cite{jabuka:delta}. It now follows from Theorem~\ref{T:A} that $\Lef(\tau_*)$ is a knot concordance invariant. Therefore, if $K_n$ is a knot for which $\Lef(\tau_*) \neq 0$, then the same is true for any knot $K$ in the concordance class of $K_n$. In particular, the $n$-fold cyclic branched cover of $K$ is not an $L$-space. 

All that remains to prove the theorem is to find a knot $K_n$ with $\Lef(\tau_*) \neq 0$. Pick relatively prime integers $q$ and $r$ both of which are greater than or equal to $2$ and are relatively prime with $p$; we will exclude the triple $(2,3,5)$ to avoid dealing with the exceptional case of the Poincar\'e homology sphere. The $n$--fold cyclic branched cover of the right-handed torus knot $T(q,r)$ is the Brieskorn homology sphere
\[
\Sigma(n,q,r)\;=\; \{\, x^n + y^q + z^r = 0\,\}\,\cap\, S^5
\]
with its canonical link orientation and with the covering translation $\tau (x,y,z) = (e^{2\pi i/n} x, y, z)$. The homology sphere $\Sigma (n,q,r)$ admits a fixed point free circle action $t (x,y,z) = (t^{qr}x, t^{nr} y, t^{nq}z)$ making it into a Seifert fibered manifold; see Neumann--Raymond \cite{neumann-raymond}. The covering translation $\tau$ is actually contained in this circle action: it corresponds to the choice of $t = e^{2\pi ip/n}$ for any integer $p$ such that $pqr \equiv 1 \pmod n$. This implies that $\tau$ is isotopic to the identity and that $\Lef(\tau_*) = \chi(\hmred(\Sigma(n,q,r)))$. 

%\smargin{Jianfeng: This argument needs some more details: $HM^{red}$ is not just the homology of irreducible critical points. One needs to argue that the irreducible points appear in pairs so can not be all canceled with reducible points (and we have to address the issure of transversality). I would say that it's easier to just cite the corresponding result in Heegard Floer homology. For example: Seifert fibered homology spheres with trivial Heegaard Floer homology by Eaman Eftekhary. \color{teal} I agree; see if you like the revised version}

To show that $\chi(\hmred(\Sigma(n,q,r)))\neq 0$, we will use the identification of the monopole and Heegaard Floer homology due to Kutluhan, Lee, and Taubes \cite{kutluhan-lee-taubes:HFSW-I,kutluhan-lee-taubes:HFSW-II,kutluhan-lee-taubes:HFSW-III,kutluhan-lee-taubes:HFSW-IV,kutluhan-lee-taubes:HFSW-V}, or alternatively, Colin, Ghiggini, and Honda \cite{colin-ghiggini-honda:HFECH-I,colin-ghiggini-honda:HFECH-II,colin-ghiggini-honda:HFECH-III} and Taubes \cite{taubes:HMECH}. The non-vanishing of the Euler characteristic of $\hmred(\Sigma(n,q,r))$ follows from this identification and the corresponding result in Heegaard Floer homology~\cite{can-karakurt:lattice} as well as \cite{Eftekhary,rustamov:plumbed}.
\end{proof}

For any prime power $n$, we can define a subgroup $\BL_n$ in the smooth concordance group $\CC$ generated by knots that are concordant to a knot whose $n$-fold branched cover is an $L$-space. One would expect that, for a given $n$, the group $\BL_n$ is rather small; our Theorem~\ref{conc} provides some evidence for that.

\begin{cor}\label{C:conc}
Let $n$ be a prime power. Then $\CC/\BL_n$ has a $\Z$ summand. 
\end{cor}
\begin{proof}
It follows from the proof of Theorem \ref{conc} that $\Lef(\tau_*): \CC \to \Z$ is a well-defined non-zero map. Since both $h(\Sigma,\s)$ and the Tristram--Levine signatures are homomorphisms, Theorem~\ref{T:A} implies that $\Lef(\tau_*)$ is a homomorphism as well. Thus there is a surjection from $\CC/\BL_n$ to the image of this homomorphism, which is isomorphic to $\Z$.
\end{proof}

\begin{rmk}
One could alternately deduce that the Lefschetz number is a homomorphism from the splitting formula of~\cite{LRS1} and the additivity of $\lsw$ proved in~\cite{ma:fiber}. %\smargin{I removed the sentence about change sign of Lefschetz number under concordance inverse, because this actually follows from the duality property of monopole Floer homology.}%That the Lefschetz number changes sign when one takes the concordance inverse of the knot makes use of the splitting formula, as in the proof of Corollary~\ref{C:orient} below.
%\item The paper~\cite{moy} shows how to enumerate the solutions to the Seiberg-Witten equations on $\Sigma(n,q,r)$ in terms of a certain count of lattice points in $3$-space. Coupled with known calculations for the Tristram-Levine signatures of torus knots, this gives in principle a combinatorial calculation of the $h$-invariants for the corresponding Brieskorn homology spheres. 
\end{rmk}

%In a further application in this vein, we can show that for certain knots $K$, no even order cyclic branched cover of $K$ is an $L$-space.  This is in agreement with a pattern observed in~\cite{boileau-boyer-gordon:branched-qp} that for some knots, $\Sigma_n(K)$ will not be an $L$-space for all sufficiently large $n$. 

\begin{thm}[Theorem~\ref{T:Jones} from the introduction]\label{Jones}
Let $K\subset S^{3}$ be a knot with $\det(K)=1$ and $J'(-1)\neq 0$, where $J_K (t)$ is the Jones polynomial of $K$.  Then, for any $m\geq1$, the $2m$-fold cyclic branched cover $\Sigma_{2m}(K)$ is not an $L$-space.
\end{thm}

%We remark that since $det(K)=1$, the condition that $\Sigma_{2m}(K)$ be a rational homology sphere is automatically satisfied when $m$ is a prime power.

\begin{proof}
We may assume without loss of generality that $\Sigma_{2m}(K)$ is a rational homology sphere. Since $\det(K) =1$,  $\Sigma_{2}(K)$ is a homology sphere, so this is automatic when $m$ is a prime power. Suppose $\Sigma_{2m}(K)$ is an $L$-space and apply the formula of Theorem \ref{T:A} to the covering translation $\tau$ and its square $\tau^{2}$. We obtain the formulas
\begin{align*}
h(\Sigma_{2m}(K),\s) &=\frac{1}{8}\;\sum_{j=1}^{2m-1}\sign_{j/2m}(K,S^{3})\quad\text{and} \\
h(\Sigma_{2m}(K),\s) &=m\cdot \lambda(\Sigma_{2}(K))+\frac{1}{8}\;\sum_{j=1}^{m-1}\,\sign_{j/m}(K,\Sigma_{2}(K)).
\end{align*}
Comparing them, we obtain
$$
m\cdot \lambda(\Sigma_{2}(K))\,=\,\frac{1}{8}\;\sum_{j=1}^{2m-1}\sign_{j/2m}(K,S^{3})\,-\,\frac{1}{8}\;\sum_{j=1}^{m-1}\sign_{j/m}(K,\Sigma_{2}(K)).
$$
On the other hand, we have the equality
$$
\frac 1 8\,\sum_{j=1}^{2m-1}\sign_{j/2m}(K,S^{3})\,-\,\frac 1 8\,\sum_{j=1}^{m-1}\sign_{j/m}(K,\Sigma_{2}(K))= \frac m 8\cdot \sign_{1/2}(K,S^{3}),
$$
which can be proved as follows. By pushing a Seifert surface of $K$ into the interior of the 4-ball and taking the $2$-fold and $2m$-fold branched covers, we obtain 4-manifolds $W_{2}$ and $W_{2m}$ with
\begin{gather*}
\sum_{j=1}^{2m-1}\sign_{j/2m}(K,S^{3})= \sign(W_{2m}), \\
\sum_{j=1}^{m-1}\sign_{j/m}(K,\Sigma_{2}(K))=\sign(W_{2m})-m\cdot \sign(W_{2}),\quad\text{and} 
\end{gather*}
$$
 \sign_{1/2}(K,S^{3})=\sign(W_{2}).
$$
This gives the desired formula. We therefore conclude that
\[
\lambda(\Sigma_{2}(K))=\frac{1}{8}\,\sign_{1/2}(K,S^{3}).
\]
Comparing this with Mullins's theorem~\cite{mullins:casson}
$$\lambda(\Sigma_{2}(K))=-\frac{1}{12}J'(-1)+\frac{1}{8}\sign_{1/2}(K,S^{3}),$$
we obtain $J'(-1)=0$, a contradiction.
\end{proof}

\begin{thm}[Theorem \ref{C:orient} from the introduction]
Let $X$ be the mapping torus of an orientation preserving diffeomorphism (not necessary of finite order) of a rational homology sphere. Then, for any choice of spin structure on $X$, we have
\[
-\lsw(X) = \lsw(-X).
\]
\end{thm}

\begin{proof} 
Let $X$ be the mapping torus of a finite order diffeomorphism $\tau: \Sigma \rightarrow \Sigma$ of a rational homology sphere $\Sigma$. Given a spin structure on $X$ denote by $\s$ its restriction to $\Sigma$ and observe that $\tau^{*}(\s) = \s$. Let $-\Sigma$ be the manifold $\Sigma$ with reversed orientation and denote by $-\s$ and $-\tau$ the corresponding spin structure and diffeomorphism, respectively. By the splitting theorem of $\lsw$ \cite[Theorem A]{LRS1}, we have 
\begin{gather*}
\lsw(X)=-\Lef(\tau_{*}:HM^{\red} (\Sigma,\s) \to HM^{\red} (\Sigma,\s))-h(\Sigma,\s)\quad\text{and} \\
\lsw(-X)=-\Lef((-\tau)_{*}:HM^{\red} (-\Sigma,-\s) \to HM^{\red} (-\Sigma,-\s))-h(-\Sigma,-\s).
\end{gather*}
Since $h$ is a homology cobordism invariant, it vanishes on the manifold $\Sigma\,\#\,(-\Sigma)$. It now follows from the additivity of $h$ (see \cite[Theorem 3]{froyshov:monopole}) that $h(-\Sigma,-\s)=-h(\Sigma,\s)$. Therefore, all we need is to check is that
\[
\Lef((-\tau)_{*})=-\Lef(\tau_{*}).
\]

\begin{lem}\label{L:duality}
There is a duality isomorphism $\widecheck{HM}_{a}(-\Sigma,-\s)\,\cong\, \widehat{HM}_{1+a}(\Sigma,\s)^{*}$ with respect to the canonical mod 2 grading in monopole homology.
\end{lem}

\begin{proof}
The monopole homology has two canonical gradings, the rational grading $\gr^{\Q}$ and the mod 2 grading $\gr^{(2)}$. Kronheimer and Mrowka \cite[Proposition 28.3.4]{kronheimer-mrowka:monopole} construct a duality isomorphism $\widecheck{HM} (-\Sigma,-\s)\,\to\, \widehat{HM} (\Sigma,\s)^{*}$ which maps elements of rational grading $j$ to elements of rational grading $-1 - j$. Since the relative rational grading matches modulo 2 the relative mod 2 grading, there is a universal constant $c(\Sigma,\s) \in \Q/2\Z$ such that $\gr^{\Q} = \gr^{(2)} + c(\Sigma,\s)\pmod 2$. Therefore, the above duality isomorphism maps elements of mod 2 grading $a$ to elements of mod 2 grading $1 + a + c(\Sigma,\s) + c(-\Sigma,-\s)\pmod 2$. The calculation of \cite[Lemma 2.6]{LRS1} implies that $c(\Sigma,\s) + c(-\Sigma,-\s) = 0 \pmod 2$, thereby completing the argument.
\end{proof}

Recall from \cite[(3.4)]{kronheimer-mrowka:monopole} that $HM_{a}^{\red} (\Sigma,\s) = \im j_{\,\Sigma,a}$ for the connecting homomorphism $j_{\,\Sigma,a}: \widecheck{HM}_{a}(\Sigma,\s)\rightarrow \widehat{HM}_{a}(\Sigma,\s)$. Moreover, it follows from the definition of $j_{\,\Sigma,a}$ that, under the duality isomorphisms
$$
\widecheck{HM}_{a}(-\Sigma,-\s)\,\cong\, \widehat{HM}_{1+a}(\Sigma,\s)^{*}\quad \text{and}\quad \widehat{HM}_{a}(-\Sigma,-\s)\,\cong\, \widecheck{HM}_{1+a}(\Sigma,\s)^{*}
$$
of Lemma \ref{L:duality}, the map $j_{-\Sigma,a}$ is the dual of the map $j_{\,\Sigma,1+a}$. As a consequence, we obtain an isomorphism 
$$
HM_{a}^{\red} (-\Sigma,-\s)\,\cong\, HM_{1+a}^{\red} (\Sigma,\s)^{*}.
$$
With respect to this isomorphism, $(-\tau)_{*}$ is the dual of $\tau_{*}$, which implies that $\Lef((-\tau)_{*})=-\Lef(\tau_{*})$.
\end{proof}

\end{document}